\documentclass[11pt,a4paper]{article}
\usepackage{indentfirst}
\usepackage{amsfonts}
\usepackage{amssymb}
\usepackage{mathrsfs}
\usepackage{amsmath}
\usepackage{amsthm}
\usepackage{enumerate}
\usepackage{cite}
\usepackage{mathrsfs}
\allowdisplaybreaks
\usepackage{geometry}
\usepackage{ifpdf}
\usepackage{ifpdf}
\ifpdf
  \usepackage[colorlinks=true,linkcolor=blue,citecolor=red, final,hyperindex]{hyperref}
\else
  \usepackage[colorlinks,final,hyperindex,hypertex]{hyperref}
\fi

\usepackage{txfonts}
\usepackage{indentfirst}
\usepackage{latexsym}
\usepackage[all]{xy}

\newcommand{\be }{\begin{equation}}
\newcommand{\ee }{\end{equation}}

\newcommand{\huaF}{\mathcal{F}}
\newcommand{\huaG}{\mathcal{G}}

\newcommand{\frkR}{\mathfrak R}

\newcommand{\Courant}[1]{\left\llbracket  #1\right\rrbracket }

\newcommand{\Id}{{\rm{Id}}}

\newcommand{\br}[1]{   [ \cdot,    \cdot  ]   }
\newcommand{\ltp}[1]{\Courant{\cdot,\cdot,\cdot}}

\newcommand{\Hom}{\mathrm{Hom}}

\newcommand{\gl}{\mathfrak {gl}}

\newcommand{\ad}{\mathrm{ad}}

\newtheorem{defn}{Definition}[section]
\newtheorem{thm}[defn]{Theorem}
\newtheorem{lem}[defn]{Lemma}
\newtheorem{prop}[defn]{Proposition}
\newtheorem{cor}[defn]{Corollary}
\newtheorem{ex}[defn]{Example}
\newtheorem{re}[defn]{Remark}

\begin{document}
	
\title{\sf Twisting  $\mathcal{O}$-operators by $(2,3)$-Cocycle of  Hom-Lie-Yamaguti  Algebras   with    Representations}

\author{ 
 Sami Mabrouk $^{1}$\footnote{E-mail: mabrouksami00@yahoo.fr },\ \
Sergei Silvestrov $^{2}$\footnote{ E-mail: sergei.silvestrov@mdu.se \text{(Corresponding author)}},\ \ Fatma  Zouaidi$^{3}$
    \footnote {  E-mail:  zouaidifatmazouaidi@gmail.com}
\\ \\
$^{1}${\small  University of Gafsa, Faculty of Sciences Gafsa, 2112 Gafsa, Tunisia}\\
$^{2}${\small M\"{a}lardalen University,
Division of Mathematics and Physics,} \\
{\small
School of Education, Culture and Communication,} \\
\nopagebreak
{\small \hspace{0.5 cm} Box 883, 72123 V\"{a}ster{\aa}s, Sweden}\\$^{3}${\small University of Sfax, Faculty of Sciences Sfax,  BP 1171, 3038 Sfax, Tunisia} 
}

\date{}
\maketitle
\begin{abstract}
In this paper, we first introduce the notion of twisted $\mathcal O$-operators on a  Hom-Lie-Yamaguti  algebra by a given $(2,3)$-cocycle with coefficients in a  representation. We show that a twisted $\mathcal O$-operator induces a Hom-Lie-Yamaguti structure.  We also introduce the notion of a weighted Reynolds operator on a Hom-Lie-Yamaguti algebra, which can serve as a special case of twisted $\mathcal O$-operators on Hom-Lie-Yamaguti algebras. Then, we  define a cohomology of twisted $\mathcal O$-operator on  Hom-Lie-Yamaguiti algebras with coefficients in a  representation. Furthermore, we introduce and study the Hom-NS-Lie-Yamaguti algebras as the underlying structure of the twisted $\mathcal O$-operator on Hom-Lie-Yamaguti algebras. Finally, we investigate the twisted $\mathcal O$-operator on Hom-Lie-Yamaguti algebras induced by the twisted $\mathcal O$-operator on a Hom-Lie algebras. 
\end{abstract}

\tableofcontents

\smallskip

\smallskip
\section{Introduction}
\textbf{Hom-algebras.} 
The area of Hom-algebras was initiated in the work of Hartwig, Larsson and Silvestrov in \cite{HartLarSil20032006DefLiealgsigmaderiv}, where the general quasi-deforma\-tions and discretizations of Lie algebras of vector fields using more general $\sigma$-derivations (twisted derivations) in place of ordinary derivations along with a general method for construction of deformations of Witt and Virasoro type algebras have been developed, motivated initially by specific examples of $q$-deformed Jacobi identities in the $q$-deformed (quantum) algebras in mathematical physics associated to $q$-difference operators and corresponding $q$-deformations of differential calculi
\cite{
AizawaSaito,
ChaiElinPop,
ChaiIsLukPopPresn,
ChaiKuLuk,
ChakrabartiJagannathan1992pqdefVirasoro,
Hu,
LiuKQuantumCentExt,
LiuKQCharQuantWittAlg,
LiuKQCharQuantWittAlgrootsunitiy,
LiuKQPhDthesis}.
These $q$-deformed algebras, with $q$-deformed Jacobi identities associated to $q$-difference operators and corresponding $q$-deformations of differential calculi,
serve as initial examples of more general quantum deformations of differential calculus and corresponding algebras, obtained by replacing the
usual derivation by general $(\sigma,\tau)$-derivations, which satisfy modified Jacobi identities deformed by some linear maps in some special ways as was shown in \cite{HartLarSil20032006DefLiealgsigmaderiv, LarssonSilvJA2005QuasiHomLieCentExt2cocyid, LarssonSilvQuasidefsl2-LUPr04Arx05CA07}, where also general algebras satisfying such identities where introduced and called Hom-Lie algebras and quasi-Hom-Lie algebras.
Furthermore, the general quasi-Lie algebras, which contain as subclasses the quasi-Hom-Lie algebras and Hom-Lie algebras,
as well as general color quasi-Lie algebras, which also include Hom-Lie superalgebras and Hom-Lie color algebras as subclasses,
were first introduced in
\cite{HartLarSil20032006DefLiealgsigmaderiv,LarssonSilvJA2005QuasiHomLieCentExt2cocyid,LarssonSilvestrov200405QuasiLiealg,
LarssonSilvestrovGradedquasiLiealg2005ChJPhys,SigurdssonSilvestrovCzech:grquasiLieWitt,SigSilvGLTbdSpringer2009HomLieWitttype}.
With these works, the area of Hom-algebras has started.
The Hom-associative algebras play the role of associative algebras in the Hom-Lie setting.
They were first introduced in \cite{MakhlSilv200608JGLTAhomstructure}, where it is shown that the commutator bracket defined from the multiplication in a Hom-associative algebra yields a Hom-Lie algebra, that is, Hom-associative algebras are Hom-Lie admissible.
Furthermore, in \cite{MakhlSilv200608JGLTAhomstructure}, the more general $G$-Hom-associative algebras, which include as subclasses the Hom-associative algebras, as well as the Hom-Vinberg and Hom-pre-Lie algebras extending to the Hom-algebra structures  setting the Vinberg algebras and pre-Lie algebras, were introduced and also shown to be Hom-Lie admissible.
The adjoint functor from the category of Hom-Lie algebras to the category of 
Hom-associative algebras and the enveloping algebra were considered in  \cite{Yau:EnvLieAlg}.
The fundamentals of formal deformation theory and associated cohomology
structures for Hom-Lie algebras have been considered first in
\cite{MakhSilvForumMath2010HomDeform}. The elements of homology for Hom-Lie
algebras have been developed in \cite{DY2009JLT2007arxHomAlgsHomol}. In \cite{MakhSil200709HomHopfGLTBSpringer} and \cite{MakhSilvJAA2010HomAlgHomCoalg},
the theory of Hom-coalgebras and related structures is developed. 
Representations and cohomologies of Hom-Lie algebras were considered in \cite{Jebhi, Sheng}, and cohomologies adapted to central extensions and deformations
were studied in \cite{Jebhi, Sheng}. Constructions of $n$-ary Hom-Nambu-Lie algebras from Hom-Lie algebras
were addressed in \cite{ArnMakhSil2014-StrCoh3LieIndLie, ArnMakhSil2010-TernaryHomNambuLieIndHomLie, ArnMakhSil2011-constr-nLienAryHomNambuLie, kms:nhominduced, kms:solvnilpnhomlie2020, AmmMabMakh,D.Y}. The universal enveloping algebras and Poincar'e-Birkhoff-Witt theorem for involutive Hom-Lie algebras were obtained in \cite{GuoZhangZheng}. Hom-Lie algebras with derivations were
studied in \cite{LiWang}. 

In \cite{GaparayiIssa-IJA2012-TwistedgenLieYamalgs}, Gaparayi and Issa introduced the concept of Hom-Lie-Yamaguti algebras, which can be viewed as a Hom-type generalization of Lie-Yamaguti algebras. Recently, in \cite{MaChenLin2}, Lin and Chen introduced the quasi-derivations of Lie-Yamaguti algebras; the Hom version is studied in \cite{MaChenLin}. In \cite{ZhangLi}, Zhang and Li introduced the representation and cohomology theory of Hom-Lie-Yamaguti algebras and studied the deformations and extensions of Hom-Lie-Yamaguti algebras as an application, generalizing the results of \cite{ZhangLi2}. In \cite{ZhangHanBi}, the authors introduced the notion of crossed modules for Hom-Lie-Yamaguti algebras and studied their construction of Hom-Lie-Yamaguti algebras.

\textbf{Yamaguti cohomology.} The cohomology theory of Lie–Yamaguti algebras was
established in Yamaguti \cite{Yamaguti}. In \cite{ZhangLi2}, Zhang and Li study the deformation and extension theory of Lie–Yamaguti algebras and 
prove that a $1$-parameter infinitesimal deformation of a Lie–Yamaguti algebra  
corresponds to a Lie–Yamaguti algebra of deformation type and a $(2,3)$-cocycle in the adjoint representation. The representation and cohomology theory of Hom-Lie–Yamaguti algebras are introduced in \cite{ZhangLi}. The authors study the deformation and extension of the Hom-Lie-Yamaguti algebras and prove that the abelian extensions of Hom-Lie–Yamaguti algebras are classified by the $(2,3)$-cohomology group.

\textbf{Reynolds operator.} Reynolds operators occurred for the first time in O. Reynolds’ famous study of turbulence theory into fluid dynamics (\cite{Reynolds}). In turbulent flow models of fluid dynamics, especially in the Reynolds-averaged Navier-Stokes equations, Reynolds operators often take the average over the fluid flow under the group of time translations. Subsequently, the Reynolds operator was named
in \cite{KamF}, where the operator was considered as a mathematical subject in general.
See \cite{Miller1, Miller2, Rota} for more studies of Reynolds operators. In particular, the free Reynolds algebras were given in \cite{ZhangGuo}. Recently, A. Das in \cite{Das2} introduced a generalization of Reynolds operators on an associative algebra and a
Lie algebra,  called twisted Rota-Baxter operators, generalized Reynolds operators or twisted $\mathcal O$-operators. Twisted $\mathcal O$-operators on other algebraic structures have also been widely studied, including 3-Lie algebras \cite{Atef}, $3$-Hom-Lie algebras \cite{LiWang2}, Hom-Lie algebras \cite{Xu} and Lie triple systems \cite{Rahma}.

\textbf{Main results and layout of the paper.} Our aim in this paper is to consider twisted $\mathcal O$-operators on a Hom-Lie-Yamaguti  algebras. 
We show that a  twisted $\mathcal O$-operators on a Hom-Lie-Yamaguti  algebra $T$ induce a new Hom-Lie-Yamaguti  algebra structure and there is a suitable representation of it. The corresponding
Yamguti cohomology is called the cohomology of the twisted $\mathcal O$-operators. As an application of cohomology, we study deformations of twisted $\mathcal O$-operators on a Hom-Lie-Yamaguti  algebra $T$. 
Finally, we introduce a new algebraic structure, called NS-Hom-Lie-Yamaguti  algebra algebras. We show
that the Hom-NS-Lie-Yamaguti algebra splits the Hom-Lie-Yamaguti algebra and the underlying structure of a twisted $\mathcal O$-operator. 

This paper is organized as follows. In Section \ref{Sec2}, we review some basic notions and facts about Hom-algebras, such as Hom-Lie-Yamaguti algebras, and their representations and cohomology theory. In section \ref{Sec3},  we introduce the notion of twisted $\mathcal O$-operators, in particular Reynolds operators on a Hom-Lie-Yamaguti  algebra, and we give some constructions. In Section \ref{Sec4}, we develop the
cohomology of twisted $\mathcal O$-operators on a Hom-Lie-Yamaguti  algebra with coefficients in a suitable representation. In Section \ref{Sec5}, we study the deformations of twisted $\mathcal O$ operators through the theory of cohomology. Finally, we introduce the notion of NS-Hom-Lie-Yamaguti  algebra as a generalization of (pre-)Hom-Lie-Yamaguti algebra, which is the underlying algebraic structure of twisted $\mathcal O$-operators on a Hom-Lie-Yamaguti  algebra. Moreover, we show that Hom-Lie algebras, Hom-NS-Lie algebras, Hom-Lie-Yamaguti algebras and Hom-NS-Lie-Yamaguti algebras are closely
related in Section \ref{Sec6}.

\section{Basics on Hom-Lie-Yamaguti algebras}\label{Sec2}
All linear spaces in the article are assumed to be over a field $\mathbb{K}$ of characteristic $0$ and finite-dimensional.

In this section, we recall some basic notions about Hom-algebras, such as Hom-Lie algebras,  Hom-Lie triple systems, Hom-Lie-Yamaguti algebras, representations and their cohomology theory.
Henceforth, $\underset{x,y,z}{\circlearrowleft}$ denotes the summation over cyclically permuted $(x,y,z)$. 
\begin{defn}[\hspace{-0.2pt}\cite{HartLarSil20032006DefLiealgsigmaderiv,LarssonSilvestrov200405QuasiLiealg,MakhlSilv200608JGLTAhomstructure}] 
A Hom-Lie  algebra is a triple $(A,[\cdot,\cdot],\alpha)$ consisting of a linear space $A$, a bilinear map ("bracket") 
$[\cdot,\cdot]:A\times A\rightarrow A$ and a linear map $\alpha:A\rightarrow A$ satisfying for all $x,y,z\in A$,   
\begin{eqnarray}
& [x,y]=-[y,x] & \  
\begin{minipage}{3cm} \centering 
{\small Skew-symmetry} \\ 
{\small identity} 
\end{minipage} 
\label{HomLieskewsym}
\\ \label{hj}
& \underset{x,y,z}{\circlearrowleft} [\alpha(x),[y,z]] = [\alpha(x),[y,z]]+[\alpha(y),[z,x]]+[\alpha(z),[x,y]]= 0.& \  
\begin{minipage}{2.5cm} \centering 
{\small Hom-Lie Jacobi} \\ 
{\small identity} 
\end{minipage}
\label{HomLieJacobiIdentity}
\end{eqnarray}
\end{defn}
    
Using skew-symmetry and bilinearity of the bracket, the Hom-Lie Jacobi identity \eqref{HomLieJacobiIdentity} for Hom-Lie  algebras can be written as
 $ [\alpha(x),[y,z]]=[[x,y],\alpha(z)]+[\alpha(y),[x,z]].$ 
     When $\alpha= Id$, Hom-Lie algebra is a Lie algebra since the Hom-Jacobi identity becomes the Jacobi identity for Lie algebras. An algebra is a Lie algebra if and only if the identity map $\alpha=Id$ is in the linear space of all Hom-Lie structures on the algebra consisting of all the linear maps $\alpha$ making the algebra into a Hom-Lie algebra.  
     
\begin{defn} \label{Def-multiplHomLie}  The  Hom-Lie algebra $(A,[\cdot,\cdot],\alpha)$ is called multiplicative if $\alpha\circ[\cdot,\cdot]=[\cdot,\cdot]\circ\alpha^{\otimes2}$ meaning for all $x,y\in A$, 
\begin{equation} \label{HLAmultipl} \alpha([x,y])=[\alpha(x),\alpha(y)],\end{equation} 
that is, in other words, if the linear map $\alpha:A\rightarrow A$ is an algebra endomorphism of $(A,[\cdot,\cdot])$.
\end{defn}

\begin{defn} \label{Def-multiplHomLie} Let $V$ be a linear space  and $\beta\in End(V)$ is a linear operator on $V$. A representation of a Hom-Lie algebra $(A,[\cdot,\cdot],\alpha )$ on $(V,\beta)$ is the linear map $\rho : A \to End(V) $, or more precisely is the triple $(V,\rho,\beta)$, such that for all $x,y \in A$,
\begin{equation} 
\begin{array}{l}
\rho(\alpha(x))\circ \beta=\beta\circ \rho(x),\\
\rho([x,y]) \circ\beta= \rho(\alpha(x))\rho(y)-\rho(\alpha(y))\rho(x).
\end{array}
\label{RLIE} 
\end{equation}
\end{defn}

If $\alpha=\beta=Id$, then Definition \ref{Def-multiplHomLie} reduces to the definition of the representations of Lie algebras. However, if $\alpha=Id$, but $\beta\neq Id$, then one gets a more general notion of representation of Lie algebra dependent on the map $\beta$ commuting with all operators of the representation.    

Let $(V,\rho,\beta)$ be a representation of a Hom-Lie algebra $(A,[\cdot,\cdot] , \alpha)$. A skew-symmetric bilinear map  $\huaF:A \times A  \to V $ is a $2$-cocycle, if it  satisfies for all $x,y,z \in A$,
\begin{align}
\label{CoLie}&\underset{x,y,z}{\circlearrowleft}\big(\rho(\alpha(x)) \huaF(y,z) + \huaF(\alpha(x),[y,z])\big) =0.
\end{align}
\begin{defn}[\hspace{-0.2pt}\cite{GaparayiIssa-IJA2012-TwistedgenLieYamalgs}]\label{LY}
A {\bf Hom-Lie-Yamaguti algebra}  
$(A,[\cdot,\cdot],\Courant{\cdot,\cdot,\cdot},\alpha)$ consists of a linear space $A$ equipped with a linear map $\alpha:A\rightarrow A$, a bilinear bracket $[\cdot,\cdot]: A \times A  \to A $ and  a trilinear bracket  $\Courant{\cdot,\cdot,\cdot}: A\times A \times A  \to A $, satisfying for all $x,y,z,w,t \in A$, 
\begin{align}
&\label{HLHLYskewsym} [x,y]=-[y,x],\quad \Courant{x,y,z}=-\Courant{y,x,z}, \\
&\label{LY1} \underset{x,y,z}{\circlearrowleft}[[x,y],\alpha(z)]+\underset{x,y,z}{\circlearrowleft}\Courant{x,y,z}=0, \\
&\label{LY2}\underset{x,y,z}{\circlearrowleft}\Courant{[x,y],\alpha(z),\alpha(w)}=0,
\\
 &\label{LY3}\Courant{\alpha(x),\alpha(y),[z,w]}=[\Courant{x,y,z},\alpha^2(w)]+[\alpha^2(z),\Courant{x,y,w}], \\
&\Courant{\alpha^2(x),\alpha^2(y),\Courant{z,w,t}}
=\Courant{\Courant{x,y,z},\alpha^2(w),\alpha^2(t)}
+\Courant{\alpha^2(z),\Courant{x,y,w},\alpha^2(t)} \nonumber \\ &\hspace{9cm}+\Courant{\alpha^2(z),\alpha^2(w),\Courant{x,y,t}}.
\label{HLYa2homLeib}
\end{align}
\end{defn}
When $\alpha=\Id$, a Hom-Lie-Yamaguti algebra reduces to a Lie-Yamaguti algebra.
\begin{defn}\label{def-HomLieYamaguti}
A Hom-Lie-Yamaguti algebra   
$(A,[\cdot,\cdot],\Courant{\cdot,\cdot,\cdot},\alpha)$ is called  multiplicative if \\$\alpha\circ[\cdot,\cdot]=[\cdot,\cdot]\circ\alpha^{\otimes2}$ and $\alpha\circ\Courant{\cdot,\cdot,\cdot}=
\Courant{\cdot,\cdot,\cdot}\circ\alpha^{\otimes3}$, that is, if for all $x,y,z\in A$, 
\begin{equation} \label{HLYmultpl} 
\alpha([x,y])=[\alpha(x),\alpha(y)], \quad 
\alpha( \Courant{x,y,z})=\Courant{\alpha(x),\alpha(y),\alpha(z)}.
\end{equation}
\end{defn}



\begin{re}
For $[\cdot,\cdot]=0$, the  (multiplicative) Hom-Lie-Yamaguti algebra  $(A, [\cdot,\cdot], \Courant{\cdot,\cdot,\cdot},\alpha)$ reduces to a (multiplicative) Hom-Lie triple system  with respect to $\alpha^2$. On the other hand, if $\Courant{\cdot,\cdot,\cdot}=0$,  then the (multiplicative) Hom-Lie-Yamaguti algebra  reduces to a (multiplicative) Hom-Lie algebra  $(A, [\cdot,\cdot],\alpha)$.
\end{re}

Let $(A_1,[\cdot,\cdot]_1,\Courant{\cdot,\cdot,\cdot}_1,\alpha_1)$ and $(A_2,[\cdot,\cdot]_2,\Courant{\cdot,\cdot,\cdot}_2,\alpha_2)$ be two Hom-Lie-Yamaguti algebras. A linear map
$\phi:A_1\to A_2$ is called the morphism of Hom-Lie-Yamaguti algebras if for all $x,y,z\in A_1$,
$$\phi([x,y]_1)=[\phi(x),\phi(y)]_2\quad\text{and}\quad \phi(\Courant{x,y,z}_1)=\Courant{\phi(x),\phi(y),\phi(z)}_2.$$
\begin{thm} \label{induced}
If $(A,[\cdot,\cdot],\alpha)$ is a  multiplicative Hom-Lie algebra and for all $x,y, z \in A$, 
   \begin{align*}
     \Courant{x,y,z}_T=[[x,y],\alpha(z)],
 \end{align*}  
then $(A,[\cdot,\cdot],\Courant{\cdot,\cdot,\cdot}_T,\alpha)$ 
is the induced multiplicative Hom-Lie-Yamaguti algebra.
\end{thm}
\begin{proof}The left-antisymmetry of $\Courant{\cdot,\cdot,\cdot}$ follows from the anti-symmetry of $[\cdot,\cdot]$. Hence the first part follows. For the second part and using the multiplicativity of $\alpha$ and \eqref{HomLieJacobiIdentity}, for all $x,y,z \in A$, 
\begin{align*}
    & \underset{x,y,z}{\circlearrowleft}[[x,y],\alpha(z)]+\underset{x,y,z}{\circlearrowleft}\Courant{x,y,z}=\underset{x,y,z}{\circlearrowleft}\Courant{x,y,z}=\underset{x,y,z}{\circlearrowleft}[[x,y],\alpha(z)]\\ 
    &\stackrel{\eqref{HomLieskewsym}}{=}-\underset{x,y,z}{\circlearrowleft}[\alpha(z),[x,y]]\stackrel{\eqref{HomLieJacobiIdentity}}{=} 0.
\end{align*}
Similarly, 
$\underset{x,y,z}{\circlearrowleft}\Courant{[x,y],\alpha(z),\alpha(w)}=\underset{x,y,z}{\circlearrowleft}[[[x,y],\alpha(z)],\alpha^2(w)]=[\underset{x,y,z}{\circlearrowleft}[[x,y],\alpha(z)],\alpha^2(w)]\stackrel{hj}{=} 0.
$
It holds that       
\begin{align*}
    &\Courant{\alpha(x),\alpha(y),[z,w]}-[\Courant{x,y,z},\alpha^2(w)]-[\alpha^2(z),\Courant{x,y,w}]\\&=[[\alpha(x),\alpha(y)],\alpha([z,w])]-[[[x,y],\alpha(z)],\alpha^2(w)]-[\alpha^2(z),[[x,y],\alpha(w)]]\\
    &\stackrel{\eqref{HLYmultpl}}{=} [[\alpha(x),\alpha(y)],\alpha([z,w])]-[[[x,y],\alpha(z)],\alpha(\alpha(w))]-[\alpha(\alpha(z)),[[x,y],\alpha(w)]]\\
    &\stackrel{\eqref{HomLieskewsym}}{=} [[\alpha(x),\alpha(y)],\alpha([z,w])]-[\alpha(\alpha(w)),[\alpha(z),[x,y]]]-[\alpha(\alpha(z)),[[x,y],\alpha(w)]] \\ 
    &\stackrel{\eqref{HomLieJacobiIdentity}}{=} [[\alpha(x),\alpha(y)],\alpha([z,w])]+[\alpha([x,y]),[\alpha(w),\alpha(z)]]\\
    &\stackrel{\eqref{HLAmultipl}-\eqref{HLYmultpl}}{=} [[\alpha(x),\alpha(y)],[\alpha(z),\alpha(w)])]-[[\alpha(x),\alpha(y)],[\alpha(z),\alpha(w)]]=0.
\end{align*}
Finally, using the multiplicativity condition  we obtain  
\begin{eqnarray*}
    \Courant{\alpha^2(x),\alpha^2(y),\Courant{z,w,t}}&=&[[\alpha^2(x),\alpha^2(y)],[[\alpha(z),\alpha(w)],\alpha^2(w)]]\\&\stackrel{\eqref{HomLieJacobiIdentity}\eqref{HLYmultpl}}{=}&-[\alpha[\alpha(z),\alpha(w)],[\alpha^2(t),[\alpha(x),\alpha(y)]]]-[\alpha^3(t),[[\alpha(x),\alpha(y)],[\alpha(z),\alpha(w)]]]\\&\stackrel{\eqref{HomLieskewsym}}{=}&[\alpha[\alpha(z),\alpha(w)],[[\alpha(x),\alpha(y)],\alpha^2(t)]]+[[[\alpha(x),\alpha(y)],[\alpha(z),\alpha(w)]],\alpha^3(t)]\\&\stackrel{\eqref{HomLieJacobiIdentity}-\eqref{HLYmultpl}}{=}&[[\alpha^2(z),\alpha^2(w)],\alpha[[x,y],\alpha(t)]]+[[[[x,y],\alpha(z)],\alpha^2(w)],\alpha^3(t)]\\&&+[[\alpha^2(z),[[x,y],\alpha(w)]],\alpha^3(t)]
\\&=&\Courant{\alpha^2(z),\alpha^2(w),\Courant{x,y,t}}+
\Courant{\Courant{x,y,z},\alpha^2(w),\alpha^2(t)}\\&&
+\Courant{\alpha^2(z),\Courant{x,y,w},\alpha^2(t)}.
\end{eqnarray*}
Hence, the ternary Hom–Nambu identity \eqref{HLYa2homLeib}  holds.
Thus, $(A,[\cdot,\cdot],\Courant{\cdot,\cdot,\cdot}_T,\alpha)$ 
is the multiplicative Hom-Lie-Yamaguti algebra.
\end{proof}
\begin{ex}It is shown in {\rm \cite{GaparayiIssa-IJA2012-TwistedgenLieYamalgs}},  that 
every Lie-Yamaguti algebra $(A, [\cdot,\cdot],\Courant{\cdot,\cdot,\cdot})$ can be twisted via a morphism $\alpha$ of $A$
into a multiplicative Hom-Lie Yamaguti algebra $(A,\alpha\circ[\cdot,\cdot],\alpha^2\circ\Courant{\cdot,\cdot,\cdot},\alpha)$. 
\end{ex}
Now we give the definition of a representation of a Hom-Lie-Yamaguti algebra.
\begin{defn}[\hspace{-3pt} \cite{ZhangLi}] \label{repHlieY}
Let $(A,[\cdot,\cdot],\Courant{\cdot,\cdot,\cdot },\alpha)$ be a Hom-Lie-Yamaguti algebra and $(V,\beta)$ be a Hom-linear space (that is a linear space with a linear map $\beta:V\rightarrow V)$. A {\bf representation } of $(A,[\cdot,\cdot],\Courant{\cdot,\cdot,\cdot },\alpha)$ on $(V,\beta)$ consists of a linear map $\rho:A \to gl(V)$ and a bilinear map $\theta:A\times A \to gl(V)$ such that for all $x_1,x_2,x_3,y_1,y_2,y_3 \in A$,
\begin{align}
&\label{RL1}\rho(\alpha(x_1))\circ \beta=\beta\circ\rho(x_1),
\quad\theta(\alpha(x_1),\alpha(x_2))\circ \beta=\beta\circ\theta(x_1,x_2),\\
&\label{RL5}D_{\rho,\theta}([x_1,x_2],\alpha(x_3))+D_{\rho,\theta}([x_2,x_3],\alpha(x_1))+D_{\rho,\theta}([x_3,x_1],\alpha(x_2))=0,\\
&\label{RL6}\theta([x_1,x_2],\alpha(y_1))\circ \beta=\theta(\alpha(x_1),\alpha(y_1))\rho(x_2)-\theta(\alpha(x_2),\alpha(y_1))\rho(x_1),\\
&\label{RL7}D_{\rho,\theta}(\alpha(x_1),\alpha(x_2))\rho(y_2)=\rho(\alpha^2(y_2))D_{\rho,\theta}(x_1,x_2)+\rho(\Courant{x_1, x_2, y_2})\circ\beta,\\
&\label{RL8}\theta(\alpha(x_1),[y_1, y_2])\circ \beta=\rho(\alpha^2(y_1))\theta(x_1, y_2)-\rho(\alpha^2(y_2))\theta(x_1,y_1),\\ \nonumber
&\label{RL9}D_{\rho,\theta}(\alpha^2(x_1),\alpha^2(x_2))\theta(y_1,y_2)\\
&=\theta(\alpha^2(y_1),\alpha^2(y_2))D_{\rho,\theta}(x_1,x_2)
+\theta(\Courant{x_1,x_2,y_1},\alpha^2(y_2))\circ\beta^2+\theta(\alpha^2(y_1),\Courant{x_1,x_2,y_2})\circ\beta^2,\\ \nonumber 
&\label{RL10}\theta(\alpha^2(x_1),\Courant{y_1, y_2, y_3})\circ\beta^2\\
&=\theta (\alpha^2(y_2), \alpha^2(y_3))\theta(x_1,y_1)
    -\theta (\alpha^2(y_1), \alpha^2(y_3))\theta(x_1,y_2) + D_{\rho,\theta} (\alpha^2(y_1), \alpha
    ^2(y_2))\theta(x_1,y_3),
\end{align}
where the bilinear map $D_{\rho,\theta}:A\times A \to gl(V)$ is given by
\begin{eqnarray}
 D_{\rho,\theta}(x,y):=\theta(y,x)-\theta(x,y)+\rho(\alpha(x))\rho(y)-
 \rho(\alpha(y))\rho(x)-\rho([x,y])\circ \beta, \quad \forall x,y \in A.\label{rep}
 \end{eqnarray}
We denote a representation of $A$ on $V$ by $(V,\rho,\theta,\beta)$. In this case, we also call $V$ to be an $A$-module.
\end{defn}
 If $\alpha=Id_A$ and $\beta=Id_V$, Definition \ref{repHlieY}  coincides with the notion of representation on a Lie-Yamaguti algebra given in \cite{Yamaguti0, Yamaguti}.  A \textbf{morphism of representation } from $V$
to $V'$ consist of a Hom-Lie-Yamaguti algebra morphism  $\phi : A \to A'$ and $\psi : V \to V'$ such that 
\begin{align}
    &\label{c1}\phi \circ \beta = \beta' \circ \psi,\\
    &\label{c3}\psi( \rho(x)u)=\rho'(\phi(x))\psi(u),\\ &\label{c2}\psi(\theta(x,y)u)=\theta'(\phi(x),\phi(y))\psi(u).
    \end{align}
\begin{re}
    Let $(A,[\cdot,\cdot],\Courant{\cdot,\cdot,\cdot},\alpha)$ be a Hom-Lie-Yamaguti algebra and $(V,\rho,\theta,\beta)$ its representation. If $\rho=0$ and the Hom-Lie-Yamaguti algebra $(A,\alpha)$ reduces to a Hom-Lie tripe system $(A,\Courant{\cdot,\cdot,\cdot},\alpha)$,  then $(V,\theta,\beta)$  is a representation of the Hom-Lie triple systems $(A,\Courant{\cdot,\cdot,\cdot},\alpha)$; If $\theta=0$, $D_{\rho,\theta}=0$ and the Hom-Lie-Yamaguti algebra $(A,\alpha)$ reduces to a Hom-Lie algebra $(A,[\cdot,\cdot],\alpha)$, then $(V,\rho,\beta)$ is a representation  of the Hom-Lie algebra $(A,[\cdot,\cdot],\alpha)$. So the above definition of a representation of a Hom-Lie-Yamaguti algebra is a natural generalization of representations of Hom-Lie algebras and Hom-Lie triple systems.
\end{re}
\begin{ex}
    Let $(A,[\cdot,\cdot],\Courant{\cdot,\cdot,\cdot},\alpha)$ be a Hom-Lie-Yamaguti algebra. We define $\ad:A \to \gl(A)$ and $\frkR :A\times A \to \gl(A)$ by $x \mapsto \ad_x$ and $(x,y) \mapsto \mathfrak{R}(x,y)$ respectively, where $\ad_x(z)=[x,z]$ and $\mathfrak{R}(x,y)z=\Courant{z,x,y}$ for all $z \in A$. Then $(A,\ad,\mathfrak{R},\alpha)$ is a representation of $A$ on itself, called the {\bf adjoint representation}.
\end{ex}
Representations of a Hom-Lie-Yamaguti algebra can be characterized by the semidirect
product Hom-Lie-Yamaguti algebras.

\begin{thm}
     Let $(A,[\cdot,\cdot],\Courant{\cdot,\cdot,\cdot},\alpha)$ be a Hom-Lie-Yamaguti algebra and $(V, \beta)$ be a Hom-linear space. Assume we have a linear 
map $\rho:A \to gl(V)$ and a bilinear map $\theta:A\times A \to gl(V)$. Then $(\rho,\theta)$ is a representation
of $A$ on $V$ if and only if $A\oplus V$ is a Hom-Lie-Yamaguti algebra under the following maps:
\begin{align*}
    &(\alpha+\beta )(x+u)=\alpha(x)+\beta(u),\\&[x+u,y+v]=[x,y]+\rho(x)v-\rho(y)u,\\&\Courant{x+u,y+v,z+w}=\Courant{x,y,z}+D(x,y)w-\theta(x,z)v+\theta(y,z)u.
\end{align*}
In this case, the Hom-Lie-Yamaguti algebra  $A\oplus V$ is called semidirect product of $A$ and $V$, which is denoted by $A\ltimes_{\rho,\theta} V$.
\end{thm}
Motivated by Yamaguti's cohomology for Lie-Yamaguti algebras, let us recall the cohomology theory on Hom-Lie-Yamaguti algebras given in \cite{ZhangLi}. Let $(A,[\cdot,\cdot],\Courant{\cdot,\cdot,\cdot},\alpha)$ be a  Hom-Lie-Yamaguti algebra and $(V,\rho,\mu,\beta)$ a representation of $A$.  Denote by $C^{n}(A, V)$, the set of $n$-linear maps  $f:A\times\cdots \times A\to V$   such that the following conditions are satisfied:
\begin{eqnarray}
\label{eq:cochain01} f(\alpha(x_1),\cdots,\alpha(x_n))&=&\beta(f(x_1,\cdots,x_n)),\\
 f(x_1,\cdots,x_{2i-1},x_{2i}\cdots,x_n)&=&0,\ \  \mbox{if}\ \  x_{2i-1}=x_{2i}
\end{eqnarray}
for $2\leq 2i\leq
n$ (if $n$ is odd, then no such "skew-symmetry" condition is assumed for the last component $x_n$ with any other component $x_j$ for $1\leq j\leq n-1.$)
The set of $n$-cochains is defined by  
\begin{align*}
& C^n_{\rm HLieY}(A,V)\triangleq C^{2n}(A, V)\times C^{2n+1}(A, V)\\ 
&\triangleq
 \begin{cases}
\Hom(\underbrace{\wedge^2A\otimes \cdots \otimes \wedge^2A}_n,V)\times \Hom(\underbrace{\wedge^2A\otimes\cdots\otimes\wedge^2A}_{n}\otimes A,V), & \forall n\geqslant 1,\\
\Hom(A,V), &n=0.
\end{cases}
\end{align*}
For the coboundary map of $p$-cochains,
for any $(\huaF,\huaG)\in C^{2n}(A, V)\times C^{2n+1}(A, V)$
the coboundary operator $\delta:(\huaF,\huaG)\mapsto (\delta_{\textrm{I}}\huaF, \delta_{\textrm{II}}\huaG)$
is a mapping from $C^{2n}(A, V)\times C^{2n+1}(A, V)$ into $C^{2n+2}(A, V)\times C^{2n+3}(A, V)$ defined as follows:
\begin{align*}
 &(\delta_{\textrm{I}}f)(x_{1}, x_{2}, \cdots, x_{2n+2})\\
&=\rho(\alpha^{2n}(x_{2n+1}))g(x_{1}, \cdots, x_{2n}, x_{2n+2}))-\rho(\alpha^{2n}(x_{2n+2}))g(x_{1}, \cdots, x_{2n+1})\\ 
&\hspace{6cm} -g(\alpha(x_{1}), \cdots, \alpha(x_{2n}), [x_{2n+1},x_{2n+2}])\\
&+\sum\limits_{k=1}^{n}(-1)^{n+k+1}D(\alpha^{2n-1}(x_{2k-1}), \alpha^{2n-1}(x_{2k}))f(x_{1}, \cdots, \hat{x}_{2k-1}, \hat{x}_{2k}, \cdots, x_{2n+2})\\
&+\sum\limits_{k=1}^{n}\sum\limits_{j=2k+1}^{2n+2}(-1)^{n+k}f(\alpha^2(x_{1}), \cdots, \hat{x}_{2k-1}, \hat{x}_{2k}, \cdots, [x_{2k-1}, x_{2k}, x_{j}], \cdots, \alpha^2(x_{2n+2})),
\\
 &(\delta_{\textrm{II}}g)(x_{1}, x_{2}, \cdots, x_{2n+3})\\
&= \theta(\alpha^{2n}(x_{2n+2}), \alpha^{2n}(x_{2n+3}))g(x_{1},\cdots, x_{2n+1})-\theta( \alpha^{2n}(x_{2n+1}), \alpha^{2n}(x_{2n+3}))g(x_{1}, \cdots, x_{2n}, x_{2n+2})\\
&+\sum\limits_{k=1}^{n+1}(-1)^{n+k+1}D(\alpha^{2n}(x_{2k-1}), \alpha^{2n}(x_{2k}))g(x_{1}, \cdots, \hat{x}_{2k-1}, \hat{x}_{2k}, \cdots, x_{2n+3})\\
&+\sum\limits_{k=1}^{n+1}\sum\limits_{j=2k+1}^{2n+3}(-1)^{n+k}g(\alpha^2(x_{1}), \cdots, \hat{x}_{2k-1}, \hat{x}_{2k}, \cdots, \Courant{x_{2k-1}, x_{2k}, x_{j}}, \cdots, \alpha^2(x_{2n+3})).
\end{align*}
When $\alpha=\Id$, one recovers Yamaguti's cohomology for Lie-Yamaguti algebras in \cite{Yamaguti}.
\begin{lem}[\hspace{-1pt}\cite{ZhangLi}]
For  $n\geq 2$, the coboundary operator satisfies $\delta^2=0$. 
\end{lem} 
Define 
the $(2n, 2n+1)$-cohomology group  of a Hom-Lie-Yamaguti algebra  $A$ with coefficients in $V$ to be the quotient space
$$H^{2n}(A, V)\times H^{2n+1}(A, V)\triangleq(Z^{2n}(A, V)\times Z^{2n+1}(A, V))/(B^{2n}(T, V)\times B^{2n+1}(T, V)).$$

 For any $\huaF \in C^1_{\rm HLY}(A,V)$, the coboundary map
$\delta:C^1_{\rm HLY}(A,V)\to C^2_{\rm HLY}(A,V), $
is defined by
\begin{eqnarray*}
\label{1cochain}
\Big(\delta_{\rm I}(\huaF)\Big)(x,y)&=&\rho(x)\huaF(y)-\rho(y)\huaF(x)-\huaF([x,y]),\\
~ \label{2cochain}\Big(\delta_{\rm II}(\huaF)\Big)(x,y,z)&=&D_{\rho,\theta}(x,y)\huaF(z)+\theta(y,z)\huaF(x)-\theta(x,z)\huaF(y)-\huaF(\Courant{x,y,z}),\quad \forall x,y, z\in A.
\end{eqnarray*}

Let $Z^{2n}(T, V)\times Z^{2n+1}(A, V)$ be the subspace of $C^{2n}(A, V)\times C^{2n+1}(A, V)$ spanned by $(\huaF, \huaG)$ such that $\delta(\huaF, \huaG)=0$.
This is called the space of cocycles. The space $B^{2n}(A, V)\times B^{2n+1}(A, V)=\delta(C^{2n-2}(T, V)\times C^{2n-1}(A, V))$ is called the space of coboundaries.
In conclusion, we obtain a cochain complex whose cohomology group is called cohomology group of a Hom-Lie-Yamaguti algebra  $A$ with coefficients in $V$.
A pair $(\huaF,\huaG)\in C^2_{\rm LieY}(A,V)\times C^{3}(A,V)$ is a  cocycle if $\delta(\huaF,\huaG)=(0,0)$
 which is equivalent to
\begin{align*}   
& 0=\delta_I(\huaF,\huaG)(x_1,x_2,x_3,x_4)=\rho(\alpha^2(x_3))\huaG(x_1,x_2,x_4)-\rho(\alpha^2(x_4))\huaG(x_1,x_2,x_3)-\huaG(\alpha(x)_1,\alpha(x_2),[x_3,x_4])\\
&\hspace{5cm} +D_{\rho,\theta}(\alpha(x_1),\alpha(x_2))\huaF(x_3,x_4)-D_{\rho,\theta}(\alpha(x_3),\alpha(x_4))\huaF(x_1,x_2) \\
&\hspace{5cm}
-\huaF(\Courant{x_1,x_2,x_3}, \alpha^2(x_4))-\huaF(\alpha^2(x_3),\Courant{x_1,x_2,x_4}),\\   
& 0=\delta_{II}(\huaF,\huaG)(x_1,x_2,x_3,x_4,z)=\theta(\alpha^2(x_4),\alpha^2(z))\huaG(x_1,x_2,x_3)-\theta(\alpha^2(x_3),\alpha^2(z))\huaG(x_1,x_2,x_4)\\
&\hspace{5cm} +D_{\rho,\theta}(\alpha^2(x_1),\alpha^2(x_2))\huaG(x_3,x_4,z)-D_{\rho,\theta}(\alpha^2(x_3),\alpha^2(x_4))\huaG(x_1,x_2,z)\\
&\hspace{5cm} -\huaG(\Courant{x_1,x_2,x_3},\alpha^2(x_4),\alpha^2(z))-
\huaG(\alpha^2(x_3),\Courant{x_1,x_2,x_4},\alpha^2(z))\\
&\hspace{5cm} +\huaG(\alpha^2(x_1),\alpha^2(x_2),\Courant{x_3,x_4,z})-\huaG(\alpha^2(x_3),\alpha^2(x_4),\Courant{x_1,x_2,z}).
\end{align*}

In \cite{ZhangLi}, the authors studies cohomology and deformation   of Hom-Lie-Yamaguti algebras. They
proved that a $(2,3)$-cocycle of a Hom-Lie-Yamaguti algebra  with
coefficients in the adjoint representation generated a $1$-parameter infinitesimal deformation of a Hom-Lie-Yamaguti algebra.
A  $(2,3)$-cocycle  on Hom-Lie-Yamaguti algebra  $(A,[\cdot,\cdot],\Courant{\cdot,\cdot,\cdot},\alpha)$ with coefficients on a represenation $(V,\rho,\theta,\beta)$ is a pair   $(\huaF,\huaG)$ satisfying
\begin{align}\label{cocycle1}
&\underset{x,y,z}{\circlearrowleft}\bigg(\huaF([x,y],\alpha(z))-\rho(\alpha(x))\huaF(y,z)+\huaG(x,y,z)\bigg)=0,\end{align}\begin{align}
\label{cocycle2}&\underset{x,y,z}{\circlearrowleft}\bigg(\theta(\alpha(x),\alpha(t))\huaF(y,z)+\huaG([x,y],\alpha(z),\alpha(t))\bigg)=0,\end{align}\begin{align}\label{cocycle3}
&D_{\rho,\theta}(\alpha(x),\alpha(y))\huaF(z,t)+\rho(\alpha^2(t))\huaG(x,y,z)-\huaF(\Courant{x,y,z},\alpha^2(t))\nonumber\\&\quad\quad\quad+\huaG(\alpha(x),\alpha(y),[z,t])-\rho(\alpha^2(z))\huaG(x,y,t)-\huaF(\alpha^2(z),\Courant{x,y,t})=0,\end{align}\begin{align} \label{cocycle4} &\nonumber\theta(\alpha^2(x)_4,\alpha^2(z))\huaG(x_1,x_2,x_3)-\theta(\alpha^2(x_3),\alpha^2(z))\huaG(x_1,x_2,x_4)+D_{\rho,\theta}(\alpha^2(x_1),\alpha^2(x_2))\huaG(x_3,x_4,z)\\ \nonumber-&D_{\rho,\theta}(\alpha^2(x_3),\alpha^2(x_4))\huaG(x_1,x_2,z)-\huaG(\Courant{x_1,x_2,x_3},\alpha^2(x_4),\alpha^2(z))-
\huaG(\alpha^2(x_3),\Courant{x_1,x_2,x_4},\alpha^2(z))\\+&\huaG(\alpha^2(x_1),\alpha^2(x_2),\Courant{x_3,x_4,z})-\huaG(\alpha^2(x_3),\alpha^2(x_4),\Courant{x_1,x_2,z})=0.
\end{align}
Note that the conditions \eqref{cocycle3} and \eqref{cocycle4} implies that $(\huaF,\huaG)$ is a cocycle.
\begin{thm} Let $(A,[\cdot,\cdot],\alpha)$ be a Hom-Lie algebra with respect a representation $(V,\rho,\beta)$. Then $(V,\theta_\rho,\beta)$ is representation of Hom-Lie-Yamaguti algebra $(A,[\cdot,\cdot],\Courant{\cdot,\cdot,\cdot}_T,\alpha)$ given in Theorem \ref{induced}, where 
\begin{equation}
\theta_\rho(x,y)=\rho(\alpha(y))\rho(x),\quad\forall x,y\in A.
\end{equation}
In addition, if
 $\huaF$ is a $2$-cocycle on a Hom-Lie algebra with respect  $(V,\rho,\beta)$. Then, $(\huaF,\huaG) $ is a $(2.3)$-cocycle of  Hom-Lie-Yamaguti algebra $(A,[\cdot,\cdot],\Courant{\cdot,\cdot,\cdot}_T,\alpha)$ given in Theorem \ref{induced} with respect  $(V,\theta_\rho,\beta)$ where \begin{equation}\huaG_{\rho,\huaF}(x,y,z)=\huaF([x,y],\alpha(z))-\rho(\alpha(z))\huaF(x,y)\quad\forall x,y,z\in A.\end{equation}
\end{thm}
\begin{proof} Note that, for all $x,y\in A$, 
\begin{eqnarray*}
 D_{\rho,\theta}(x,y)&=&\theta_\rho(y,x)-\theta_\rho(x,y)+\rho(\alpha(x))\rho(y)-
 \rho(\alpha(y))\rho(x)-\rho([x,y])\circ \beta\\&=&\rho(\alpha(x))\rho(y)-\rho(\alpha(y))\rho(x)+\rho(\alpha(x))\rho(y)-
 \rho(\alpha(y))\rho(x)-\rho([x,y])\circ \beta\\&=& -\rho([x,y])\circ \beta.\label{rep}
 \end{eqnarray*}
Thus, apply the Hom-Jacobi identity, for all $x_1,x_2,x_3\in A,$ we have
 \begin{align*}
     &D_{\rho,\theta}([x_1,x_2],\alpha(x_3))+D_{\rho,\theta}([x_2,x_3],\alpha(x_1))+D_{\rho,\theta}([x_3,x_1],\alpha(x_2))\\&=-\big(\rho([[x_1,x_2],\alpha(x_3)]+\rho([[x_2,x_3],\alpha(x_1)]+\rho([[x_3,x_1],\alpha(x_2)]) \big)\\&=-\circlearrowleft_{x_1,x_2,x_3}\rho([[x_1,x_2],\alpha(x_3)])=0.
 \end{align*}
 Similarly, we can prove the other conditions.
 Now, we add and show that $(\huaF,\huaG) $ is a $(2.3)$-cocycle, since $\huaF$ is a $2$-cocycle on a Hom-Lie algebra with respect  $(V,\rho,\beta)$, for any $x, y, z \in A $, we have
 \begin{align*}
&\underset{x,y,z}{\circlearrowleft}\bigg(\huaF([x,y],\alpha(z))-\rho(\alpha(x))\huaF(y,z)+\huaG(x,y,z)\bigg)\\&=\underset{x,y,z}{\circlearrowleft}\bigg(\huaF([x,y],\alpha(z))-\rho(\alpha(x))\huaF(y,z)+\huaF([x,y],\alpha(z))-\rho(\alpha(z))\huaF(x,y)\bigg)\\&=2\underset{x,y,z}{\circlearrowleft}\bigg(\huaF([x,y],\alpha(z))-\rho(\alpha(x))\huaF(y,z)\bigg)=0.\end{align*}
The proof is finished.
\end{proof}
\section{Twisted $\mathcal O$-operators  on Hom-Lie-Yamaguti  algebras}\label{Sec3} In this section we give the Hom version of twisted  $\mathcal O$-operators on Lie-Yamaguti algebras as a generalization of $\mathcal O$-operators on Hom-Lie-Yamaguti algebras. In particular, we study  Reynolds operators on Hom-Lie-Yamaguti algebra of any weight.
\subsection{Weighted Reynolds operators}
\begin{defn}
    Let $(A,[\cdot,\cdot],\Courant{\cdot,\cdot,\cdot},\alpha)$ a Hom-Lie-Yamaguti algebra, and $\lambda,\mu\in\mathbb{K}$. A \textbf{$(\lambda,\mu)$-weighted Reynolds operator} on Hom-Lie-Yamaguti algebra is a linear map $R:A\to A$ satisfying for all $x,y,z \in A,$
\begin{align}
R\circ \alpha&=\alpha\circ R,\\
[Rx,Ry]&=R\Big([Rx,y]+[x,Ry]+\lambda[Rx,Ry]\Big),\\
\Courant{Rx,Ry,Rz}&=R\big(\Courant{Rx,Ry,z}+\Courant{Rx,y,Rz}+\Courant{x,Ry,Rz}+\mu \Courant{Rx,Ry,Rz}\big).
\end{align}
\end{defn}
\begin{ex}
A Rota-Baxter operator of weight zero on a Hom-Lie-Yamaguti algebra $(A,[\cdot,\cdot],\Courant{\cdot,\cdot,\cdot},\alpha)$ is a linear map $R : A\rightarrow A$
satisfying:
\begin{align}
    \nonumber&R\circ\alpha=\alpha\circ R,\\&
    [Rx,Ry]=R([Rx,y]+[x,Ry]),\\&\Courant{Rx,Ry,Rz}=R\big(\Courant{Rx,Ry,z}+\Courant{Rx,y,Rz}+\Courant{x,Ry,Rz}\big),\,  \forall \; x,y,z\in A.
\end{align}

Any Rota-Baxter operator on Hom-Lie-Yamaguti algebras is a $(\lambda,\mu)$-weighted 
Reynolds operator with  $(\lambda,\mu)=(0,0)$.
\end{ex}
\begin{re}

Note that $(A,[\cdot,\cdot],\Courant{\cdot,\cdot,\cdot},\alpha,R)$ is $(\lambda,\mu)$-weighted Reynolds  Hom-Lie-Yamaguti algebra. 
If $\alpha=id$ the above definition coincides with the notion of Rota-Baxter operators on Lie-Yamaguti algebra (see {\rm \cite{Fatma}}).
If $[x,y]=0$ for all $x,y\in A$ then the  $(\lambda,\mu)$-weighted Reynolds Hom-Lie-Yamaguti algebra  $(A, [\cdot,\cdot], \Courant{\cdot,\cdot,\cdot},\alpha,R)$ reduces to a $\mu$-weighted Reynolds Hom-Lie triple system  $(A,\Courant{\cdot,\cdot,\cdot},\alpha^2,R)$. On the other hand, if $\Courant{x,y,z}=0$  for all $x,y,z\in A$  then the $(\lambda,\mu)$-weighted Reynolds Hom-Lie-Yamaguti algebra  reduces to a $\lambda$-weighted Reynolds Hom-Lie algebra  $(A, [\cdot,\cdot],\alpha,R)$.
    
\end{re}

\begin{defn}
    Let $(A_1,[\cdot,\cdot]_1,\Courant{\cdot,\cdot,\cdot}_1,\alpha_1,R_1)$ and $(A_2,[\cdot,\cdot]_2,\Courant{\cdot,\cdot,\cdot}_2,\alpha_2,R_2)$ be two $(\lambda,\mu)$-weighted Reynolds  Hom-Lie-Yamaguti algebras. A {\bf morphism} from  $(A_1,[\cdot,\cdot]_1,\Courant{\cdot,\cdot,\cdot}_1,\alpha_1,R_1)$ to\\ $(A_2,[\cdot,\cdot]_2,\Courant{\cdot,\cdot,\cdot}_2,\alpha_2,R_2)$ is a linear map $\phi:A_1 \to A_2$ such that for all $x,y,z \in A$,
\begin{eqnarray*}
\phi\circ\alpha_1&=&\alpha_2\circ\phi,\\
\phi([x,y]_1)&=&[\phi(x),\phi(y)]_2,\\
~ \phi(\Courant{x,y,z}_1)&=&\Courant{\phi(x),\phi(y),\phi(z)}_2, \; and \; \phi \circ R_1=R_2 \circ \phi.
\end{eqnarray*}
\end{defn}
Let $(A,[\cdot,\cdot],\Courant{\cdot,\cdot,\cdot},\alpha,R)$ a $(\lambda,\mu)$-weighted Reynolds  Hom-Lie-Yamaguti algebra. Define the two linear maps by \begin{align}
[x,y]_R&=[Rx,y]+[x,Ry]+\lambda[Rx,Ry],\\\Courant{x,y,z}_R&=\Courant{Rx,Ry,z}+\Courant{Rx,y,Rz}+\Courant{x,Ry,Rz}+\mu \Courant{Rx,Ry,Rz}, \; \forall x,y,z \in A.
\end{align}
\begin{thm}
    $(A,[\cdot,\cdot]_R,\Courant{\cdot,\cdot,\cdot}_R,\alpha,R)$ is $(\lambda,\mu)$-weighted Reynolds Hom-Lie-Yamaguti algebra. Also, $R$ is morphism of $(\lambda,\mu)$-weighted Reynolds operators
 from $(A,[\cdot,\cdot]_R,\Courant{\cdot,\cdot,\cdot}_R,\alpha,R)$ to $(A,[\cdot,\cdot],\Courant{\cdot,\cdot,\cdot},\alpha,R)$, that is $R([x,y]_R)=[Rx,Ry]$ and $R(\Courant{x,y,z}_R)=\Courant{Rx,Ry,Rz}$.
\end{thm}
\begin{proof}For $ x,y,z,w\in A$, 
\begin{align*}
&\circlearrowleft_{x,y,\alpha(z)}\Courant{[x,y]_R,\alpha(z),\alpha(w)}_R =\circlearrowleft_{x,y,\alpha(z)}
\Courant{[Rx,y]+[x,Ry]+\lambda[Rx,Ry],\alpha(z),\alpha(w)}_R\\
=&\circlearrowleft_{x,y,\alpha(z)}\Courant{[Rx,Ry],R\alpha(z),\alpha(w)}+
\circlearrowleft_{x,y,\alpha(z)}\Courant{[Rx,Ry],\alpha(z),R\alpha(w)}\\
&\ +\circlearrowleft_{x,y,\alpha(z)}\Courant{[Rx,y],R\alpha(z),R\alpha(w)}
+\circlearrowleft_{x,y,\alpha(z)}\Courant{[x,Ry],R\alpha(z),R\alpha(w)}\\
&\ +\lambda\circlearrowleft_{x,y,\alpha(z)}\Courant{[Rx,Ry],R\alpha(z),R\alpha(w)}
+\mu\circlearrowleft_{x,y,\alpha(z)}\Courant{[Rx,Ry],R\alpha(z),R\alpha(w)}\\
=&\circlearrowleft_{x,y,\alpha(z)}\Courant{[Rx,Ry],R\alpha(z),\alpha(w)}+\circlearrowleft_{x,y,\alpha(z)}\Courant{[Rx,Ry],\alpha(z),R\alpha(w)}\\
&\ +\circlearrowleft_{x,y,\alpha(z)}\Courant{[Rx,y],R\alpha(z),R\alpha(w)}
+\circlearrowleft_{x,y,\alpha(z)}\Courant{[x,Ry],R\alpha(z),R\alpha(w)}\\
&\ +(\lambda+\mu)\circlearrowleft_{x,y,\alpha(z)}\Courant{[Rx,Ry],R\alpha(z),R\alpha(w)}=0.
\end{align*}  
The identity \eqref{LY3} holds since
\begin{align*}
& \Courant{\alpha(x),\alpha(y),[z,w]_R}_R=\Courant{R(\alpha(x)),R(\alpha(y)),[z,w]_R}+\Courant{R(\alpha(x)),\alpha(y),R([z,w]_R)}\\&\ +\Courant{\alpha(x),R(\alpha(y)),R([z,w]_R)}+\mu\Courant{R(\alpha(x)),R(\alpha(y)),R([z,w]_R)}\\=&\Courant{R(\alpha(x)),R(\alpha(y)),[Rz,w]}+\Courant{R(\alpha(x)),R(\alpha(y)),[z,Rw]}+\lambda\Courant{R(\alpha(x)),R(\alpha(y)),[Rz,Rw]}\\&\ +\Courant{R(\alpha(x)),\alpha(y),[Rz,Rw]}+\Courant{\alpha(x),R(\alpha(y)),[Rz,Rw]}+\mu\Courant{R(\alpha(x)),R(\alpha(y)),[Rz,Rw]}\\=&\Courant{R(\alpha(x)),R(\alpha(y)),[Rz,w]}+\Courant{R(\alpha(x)),R(\alpha(y)),[z,Rw]}+\Courant{R(\alpha(x)),\alpha(y),[Rz,Rw]}\\&\ +\Courant{\alpha(x),R(\alpha(y)),[Rz,Rw]}+(\lambda+\mu)\Courant{R(\alpha(x)),R(\alpha(y)),[Rz,Rw]}.\\
 & [\Courant{x,y,z}_R,\alpha^2(w)]_R=[R(\Courant{x,y,z}_R),\alpha^2(w)]+[\Courant{x,y,z}_R,R(\alpha^2(w))]+\lambda[R(\Courant{x,y,z}_R),R(\alpha^2(w))]\\=&[\Courant{Rx,Ry,Rz},\alpha^2(w)]+[\Courant{Rx,Ry,z},R(\alpha^2(w))]+[\Courant{Rx,y,Rz},R(\alpha^2(w))]\\&\ +[\Courant{x,Ry,Rz},R(\alpha^2(w))]+\mu [\Courant{Rx,Ry,Rz},R(\alpha^2(w))]+\lambda[\Courant{Rx,Ry,Rz}),R(\alpha^2(w))]\\=&[\Courant{Rx,Ry,Rz},\alpha^2(w)]+[\Courant{Rx,Ry,z},R(\alpha^2(w))]+[\Courant{Rx,y,Rz},R(\alpha^2(w))]\\&\ +[\Courant{x,Ry,Rz},R(\alpha^2(w))]+(\lambda+\mu) [\Courant{Rx,Ry,Rz},R(\alpha^2(w))].
\\
  &  [\alpha^2(z),\Courant{x,y,w}_R]_R=[R(\alpha^2(z)),\Courant{x,y,w}_R]+[\alpha^2(z),R\Courant{x,y,w}_R]+\lambda[R(\alpha^2(z)),R\Courant{x,y,w}_R]\\
  =&[R(\alpha^2(z)),\Courant{Rx,Ry,w}]+[R(\alpha^2(z)),\Courant{Rx,y,Rw}]+[R(\alpha^2(z)),\Courant{x,Ry,Rw}]\\&\ 
  +[\alpha^2(z),\Courant{Rx,Ry,Rw}]+(\lambda+\mu)[R(\alpha^2(z)),\Courant{Rx,Ry,Rw}],
\end{align*}
and hence $$\Courant{\alpha(x),\alpha(y),[z,w]_R}_R-\big[\Courant{x,y,z}_R,\alpha^2(w)\big]_R-\big[\alpha^2(z),\Courant{x,y,w}_R\big]_R=0.$$
By direct computation, it is easy to
show the other identities, we obtain that  $(A,[\cdot,\cdot]_R,\Courant{\cdot,\cdot,\cdot}_R,\alpha)$ is a Hom-Lie-Yamaguti algebra.  Note that $R$ is morphisme of Hom-Lie-Yamaguti algebras. Then, it is easy to
show that $R$ is $(\lambda,\mu)$-weighted Reynolds  Hom-Lie-Yamaguti algebra. Therefore, $R$ is a morphism of $(\lambda,\mu)$-weighted Reynolds Hom-Lie-Yamaguti algebra.
\end{proof}
\begin{prop}
    Let $R:A \to A$ be a $\lambda$-weighted Reynolds operator on Hom-Lie algebra, then $R$ is $(\lambda,2\lambda)$-weighted Reynolds operator on the induced Hom-Lie-Yamaguti algebra $(A,[\cdot,\cdot],\Courant{\cdot,\cdot,\cdot}_T,\alpha)$ defined in the Example \ref{induced}.
\end{prop}
\begin{proof}
For $ x,y, z\in A$,  
\begin{align*}
  &  \Courant{Rx,Ry,Rz}=[[Rx,Ry],R(\alpha(z))]=[R([Rx,y]+[x,Ry]+\lambda[Rx,Ry]),R\alpha(z)]\\&=R\big([[Rx,Ry],\alpha(z)]+[[Rx,y],R\alpha(z)]+[[x,Ry],R\alpha(z)]+\lambda[[Rx,Ry],R\alpha(z)]+\lambda[[Rx,Ry]),R\alpha(z)]\big)\\&=R\big([[Rx,Ry],\alpha(z)]+[[Rx,y],R\alpha(z)]+[[x,Ry],R\alpha(z)]+2\lambda[[Rx,Ry],R\alpha(z)]\big)\\
  &=R\big(\Courant{Rx,Ry,z}+\Courant{Rx,y,Rz}+
  \Courant{x,Ry,Rz}+2\lambda\Courant{Rx,Ry,Rz}\big),
\end{align*} 
and $R$ is $(\lambda,2\lambda)$-weighted Reynolds operator on induced Hom-Lie-Yamaguti algebra $(A,[\cdot,\cdot],\Courant{\cdot,\cdot,\cdot},\alpha)$.
\end{proof}
\subsection{Twisted $\mathcal{O}$-operators}
In the following, we introduce the notion of twisted $\mathcal{O}$-operator on Hom-Lie-Yamaguti algebra
and provide some new constructions.
\begin{defn}\label{defn-h-tw}
A linear map $T: V\rightarrow A$ is said to be a {\bf $(\huaF,\huaG)$-twisted $\mathcal O$-operator} or \textbf{generalized Reynolds operator} on a Hom-Lie-Yamaguti algebra $(A,[\cdot,\cdot],\Courant{\cdot,\cdot,\cdot},\alpha)$ with respected to the representation $(V,\rho,\theta,\beta)$ if $T$ satisfies
\begin{align}\label{R1}
T\circ \beta&=\alpha \circ T,\\ \label{R2}
[Tu,Tv]&=T\Big(\rho(Tu)v-\rho(Tv)u+\huaF(Tu,Tv)\Big),\\ \label{R3} \Courant{Tu,Tv,Tw}&=T\Big(D_{\rho,\theta}(Tu,Tv)w+\theta(Tv,Tw)u-\theta(Tu,Tw)v+\huaG(Tu,Tv,Tw)\Big).
\end{align}
\end{defn}
The notion of twisted $\mathcal{O}$-operator is also called twisted generalized Reynolds operator or twisted Rota-Baxter operator or twisted
Kupershmidt operator. 
If $\alpha=\Id_A$ and $\beta=\Id_V$, then the Definition  coincides with the notion of twisted $\mathcal{O}$-operator on 
Hom-Lie-Yamaguti algebra $(A,[\cdot,\cdot],\Courant{\cdot,\cdot,\cdot},\alpha)$ with respect to the representation $(V,\rho,\theta,\beta)$.
\begin{defn}\label{Tmorphism}
   A morphism of twisted $\mathcal{O}$-operators from $T$ to $T'$ consists of a morphism of representation  $(\phi,\psi)$ satisfying $T\psi=\phi T'.$
\end{defn}
\begin{ex}
   
Recall that an $\mathcal O$-operator on a Hom-Lie-Yamaguti algebra $(A,[\cdot,\cdot],\Courant{\cdot,\cdot,\cdot},\alpha)$ is a linear map $T : V\rightarrow A$
satisfying:
\begin{align}
    \nonumber&T\circ\beta=\beta\circ T,\\&
    [Tu,Tv]=T\big(\rho(Tu)v-\rho(Tv)u\big)\\&
\Courant{Tu,Tv,Tw}=T\big(D_{\rho,\theta}(Tu,Tv)w+\theta(Tv,Tw)u-\theta(Tu,Tw)v\big),\,  \forall \; x,y,z\in A.
\end{align}
Any $\mathcal{O}$-operator on a Hom-Lie-Yamaguti algebra is a $(\huaF,\huaG)$-twisted $\mathcal{O}$-operator with $(\huaF,\huaG) =(0,0)$.
\end{ex}
\begin{ex}
Any $(\lambda,\mu)$-weighted Reynolds operator   on a Hom-Lie-Yamaguti algebra is a $(\huaF,\huaG)$-twisted $\mathcal{O}$-operator where  $(\huaF,\huaG) =(\lambda[\cdot,\cdot],\mu\Courant{\cdot,\cdot,\cdot})$.
\end{ex}
In \cite{Xu} (see also \cite{Das1}),  the author  introduced a $\huaF$-twisted $\mathcal O$-operator on a Hom-Lie algebras with respect to a representation $(V,\rho,\beta)$, with a $2$-cocycle  $\huaF$ in the Chevalley-Eilenberg  cohomology, as a linear map $T  : V \to A $  satisfying 
 \begin{align*}
 &T \circ \beta= \alpha\circ T,\\
     &[T u,T v]=T\Big(\rho (T u)v-\rho (T v)u+ \huaF(T u,T v)\Big), \quad \forall u,v \in V.
 \end{align*}
 \begin{prop}
     Let $T: V\rightarrow A$ a
  $\huaF$-twisted $\mathcal{O}$-operator on a Hom-Lie algebras $(A,[\cdot,\cdot],\alpha)$ with respect to a representation $(V,\rho,\beta)$. Then $T$ is 
a $(\huaF,\huaG)$-twisted $\mathcal{O}$-operator on on the induced Hom-Lie-Yamaguti algebras given in Theorem \ref{induced}.
 \end{prop}
 \begin{proof}
   For $u,v,w \in V$, 
    \begin{multline*}
        \Courant{Tu,Tv,Tw}=[[Tu,Tv],Tw] = [T\Big(\rho (T u)v-\rho (T v)u+ \huaF(T u,T v)\Big) ,Tw]\\
        =T\Big(\rho([Tu,Tv])w-\rho(Tw)\rho (T u)v+\rho(Tw)\rho (T v)u- \rho(Tw)\huaF(T u,T v)
        +\huaF([T u,T v],Tw)\Big)\\
        =T\Big(D_{\rho,\theta}([Tu,Tv])w-\theta_\rho(Tu,T w)v+\theta_\rho(Tv,T w)u
        +\huaG_{\rho,\huaF}(T u,T v,Tw)\Big).
    \end{multline*}
    Thus, $T$ is  
a $(\huaF,\huaG)$-twisted $\mathcal{O}$-operator on on the induced Hom-Lie-Yamaguti algebras.
 \end{proof}
Let $(A,[\cdot,\cdot],\Courant{\cdot,\cdot,\cdot}, \alpha)$ be a Hom-Lie-Yamaguti algebra, $(V,\rho,\theta,\beta)$ representation, and $(\huaF,\huaG)$ be a $(2,3)$-cocycle. Define the brackets  $[\cdot,\cdot]_{\rho}$ and $\Courant{\cdot,\cdot,\cdot}_{\theta} $ for all $x,y\in A$ and $u,v\in V$ by  
\begin{align}
    [(x,u), (y,v)]_\rho&=\big([x,y],\rho(x)v-\rho(y)u+\huaF(x,y)\big),
    \\ \Courant{(x,u),(y,v),(z,w)}_\theta&=\big(\Courant{x,y,z},D_{\rho,\theta}(x,y)w+\theta(y,z)u-\theta(x,z)v+\huaG(x,y,z)\big),\label{fun}
    \\\Tilde{\alpha}(x+v)&=(\alpha+\beta)(x,v)=(\alpha(x)+\beta(v)).
\end{align}
\begin{thm}
With the above notation, $A\oplus V$ carries a Hom-Lie-Yamaguti algebra.
This is called the $(\huaF,\huaG)$-twisted semi-direct product, denoted by $A\ltimes_{\huaF,\huaG} V$.
\end{thm}
\begin{proof}
Let $x_1,x_2,x_3,x_4,x_5\in A$ and $u,v,w,t,a\in V$. 
By direct computation and using \eqref{LY1}, we obtain 
\begin{align*}
    &\underset{x_1,x_2,x_3}{\circlearrowleft} [[x_1+u,x_2+v]_\rho,(\alpha+\beta)(x_3+w)]_\rho+\underset{x_1,x_2,x_3}\circlearrowleft \Courant{x_1+u,x_2+v,x_3+w}_\theta =0.
\end{align*}
On the one hand, by \eqref{HLYa2homLeib} and \eqref{fun}, 
\begin{align*}
&\Courant{\Tilde{\alpha}^2(x_1+u),\Tilde{\alpha}^2(x_2+v),\Courant{x_3+w,x_4+t,x_5+a}}\\
&\quad =\Courant{\Tilde{\alpha}^2(x_1+u),\Tilde{\alpha}^2(x_2+v),\Courant{x_3,x_4,x_5}+D_{\rho,\theta}(x_3,x_4)a+\theta(x_4,x_5)w-\theta(x_3,x_5)t+\Psi(x_3,x_4,x_5)}\\
&\quad =\Courant{\alpha^2(x_1),\alpha^2(x_2),\Courant{x_3,x_4,x_5}}+D_{\rho,\theta}(\alpha^2(x_1),\alpha^2(x_2))D_{\rho,\theta}(x_3,x_4)a+D_{\rho,\theta}(\alpha^2(x_1),\alpha^2(x_2))\theta(x_4,x_5)w\\
&\quad\quad-D_{\rho,\theta}(\alpha^2(x_1),\alpha^2(x_2))\theta(x_3,x_5)t+D_{\rho,\theta}(\alpha^2(x_1),\alpha^2(x_2))\huaG(x_3,x_4,x_5)+\theta(\alpha^2(x_2),\Courant{x_3,x_4,x_5})\beta^2(u)\\
&\quad \quad-\theta(\alpha^2(x_1),\Courant{x_3,x_4,x_5})\beta^2(v)+\huaG(\alpha^2(x_1),\alpha^2(x_2),\Courant{x_3,x_4,x_5}).
\end{align*}
On the other hand, 
\begin{align*}
&\Courant{\Courant{x_1+u,x_2+v,x_3+w},\Tilde{\alpha}^2(x_4+t),\Tilde{\alpha}^2(x_5+a)}\\
&\quad =\Courant{\Courant{x_1,x_2,x_3}+D_{\rho,\theta}(x_1,x_2)w+\theta(x_2,x_3)u-\theta(x_1,x_3)v+\huaG(x_1,x_2,x_3),\Tilde{\alpha}^2(x_4+t),\Tilde{\alpha}^2(x_5+a)}\\
&\quad =\Courant{\Courant{x_1,x_2,x_3},\alpha^2(x_4),\alpha^2(x_5)}+D_{\rho,\theta}(\Courant{x_1,x_2,x_3},\alpha^2(x_4))\beta^2(a)+\theta(\alpha^2(x_4),\alpha^2(x_5))D_{\rho,\theta}(x_1,x_2)w\\
&\quad \quad +\theta(\alpha^2(x_4),\alpha^2(x_5))\theta(x_2,x_3)u-\theta(\alpha^2(x_4),\alpha^2(x_5))\theta(x_1,x_3)v+\theta(\alpha^2(x_4),\alpha^2(x_5))\huaG(x_1,x_2,x_3)\\&\quad \quad -\theta(\Courant{x_1,x_2,x_3},\alpha^2(x_5))\beta^2(t)+\huaG(\Courant{x_1,x_2,x_3},\alpha^2(x_4),\alpha^2(x_5)),
\\
&\Courant{\Tilde{\alpha}^2(x_3+w),\Courant{x_1+u,x_2+v,x_4+t},\Tilde{\alpha}^2(x_5+a)}\\
&\quad =\Courant{\Tilde{\alpha}^2(x_3+w),\Courant{x_1,x_2,x_4}+D_{\rho,\theta}(x_1,x_2)t+\theta(x_2,x_4)u-\theta(x_1,x_4)v+\huaG(x_1,x_2,x_4),\Tilde{\alpha}^2(x_5+a)}\\
&\quad =\Courant{\alpha^2(x_3),\Courant{x_1,x_2,x_4},\alpha^2(x_5)}+D_{\rho,\theta}(\alpha^2(x_3),\Courant{x_1,x_2,x_4})\beta^2(a)+\theta(\Courant{x_1,x_2,x_4},\alpha^2(x_5))\beta^2(w)\\
& \quad\quad -\theta(\alpha^2(x_3),\alpha^2(x_5))D_{\rho,\theta}(x_1,x_2)t-\theta(\alpha^2(x_3),\alpha^2(x_5))\theta(x_2,x_4)u+\theta(\alpha^2(x_3),\alpha^2(x_5))\theta(x_1,x_4)v\\
&\quad\quad -\theta(\alpha^2(x_3),\alpha^2(x_5))\huaG(x_1,x_2,x_4)+\huaG(\alpha^2(x_3),\Courant{x_1,x_2,x_4},\alpha^2(x_5)),
\end{align*}
Similarly, we have 
\begin{align*}
&\Courant{\Tilde{\alpha}^2(x_3+w),\Tilde{\alpha}^2(x_4+t),\Courant{x_1+u,x_2+v,x_5+a}}\\
&\quad =\Courant{\Tilde{\alpha}^2(x_3+w),\Tilde{\alpha}^2(x_4+t),\Courant{x_1,x_2,x_5}+D_{\rho,\theta}(x_1,x_2)a+\theta(x_2,x_5)u-\theta(x_1,x_5)v+\Psi(x_1,x_2,x_5)}\\
&\quad =\Courant{\alpha^2(x_3),\alpha^2(x_4),\Courant{x_1,x_2,x_5}}+D_{\rho,\theta}(\alpha^2(x_3),\alpha^2(x_4))D_{\rho,\theta}(x_1,x_2)a+D_{\rho,\theta}(\alpha^2(x_3),\alpha^2(x_4))\theta(x_2,x_5)u\\
&\quad \quad -D_{\rho,\theta}(\alpha^2(x_3),\alpha^2(x_4))\theta(x_1,x_5)v+D_{\rho,\theta}(\alpha^2(x_3),\alpha^2(x_4))\huaG(x_1,x_2,x_5)+\theta(\alpha^2(x_4),\Courant{x_1,x_2,x_5})\beta^2(w)\\
&\quad \quad -\theta(\alpha^2(x_3),\Courant{x_1,x_2,x_5})\beta^2(t)+\huaG(\alpha^2(x_3),\alpha^2(x_4),\Courant{x_1,x_2,x_5}).
\end{align*}
Then we have, according to \eqref{cocycle4} and  \eqref{HLYa2homLeib}
\begin{align*}
&\Courant{\Tilde{\alpha}^2(x_1+u),\Tilde{\alpha}^2(x_2+v),\Courant{x_3+w,x_4+t,x_5+a}}\\ 
&-\Courant{\Courant{x_1+u,x_2+v,x_3+w},\Tilde{\alpha}^2(x_4+t),\Tilde{\alpha}^2(x_5+a)}\\
&-\Courant{\Tilde{\alpha}^2(x_1+u),\Courant{x_2+v,x_3+w,x_4+t},\Tilde{\alpha}^2(x_5+a)}\\ 
&-\Courant{\Tilde{\alpha}^2(x_3+w),\Tilde{\alpha}^2(x_4+t),\Courant{x_1+u,x_2+v,x_5+a}}=0.
\end{align*}
Now let's move to check the identity \eqref{LY3}, let $x_1,x_2,x_3,x_4\in A$ and $u,v,w,t \in V$, we have
\begin{align*}
   & \Courant{\Tilde{\alpha}(x_1+u),\Tilde{\alpha}(x_2+v),[x_3+w,x_4+t]}\\
   &\quad =\Courant{\Tilde{\alpha}(x_1+u),\Tilde{\alpha}(x_2+v),[x_3,x_4]+\rho(x_3)t-\rho(x_4)w+\huaF(x_3,x_4)}\\
   &\quad =\Courant{\alpha(x_1),\alpha(x_2),[x_3,x_4]}
   +D_{\rho,\theta}(\alpha(x_1),\alpha(x_2))\rho(x_3)t
   -D_{\rho,\theta}(\alpha(x_1),\alpha(x_2))\rho(x_4)w\\
   &\quad\quad +D_{\rho,\theta}(\alpha(x_1),\alpha(x_2))\huaF(x_3,x_4)+\theta(\alpha(x_2),[x_3,x_4])\beta(u)\\
   &\quad\quad -\theta(\alpha(x_1),[x_3,x_4])\beta(v)+\huaG(\alpha(x_1),\alpha(x_2),[x_3,x_4]),\\
&
[\Courant{x_1+u,x_2+v,x_3+w},\Tilde{\alpha}^2(x_4+t)]\\
&\quad =[\Courant{x_1,x_2,x_3}+D_{\rho,\theta}(x_1,x_2)w+\theta(x_2,x_3)u-\theta(x_1,x_3)v+\huaG(x_1,x_2,x_3),\Tilde{\alpha}^2(x_4+t)]\\
&\quad =[\Courant{x_1,x_2,x_3},{\alpha}^2(x_4)]+\rho(\Courant{x_1,x_2,x_3})\beta^2(t)-\rho(\alpha^(x_4))D_{\rho,\theta}(x_1,x_2)w\\
&\quad\quad -\rho(\alpha^2(x_4))\theta(x_2,x_3)u+\rho(\alpha^2(x_4))\theta(x_1,x_3)v\\
&\quad\quad -\rho(\alpha^2(x_4))\huaG(x_1,x_2,x_3)+\huaF(\Courant{x_1,x_2,x_3},\alpha^2(x_4))
\\
 & [\Tilde{\alpha}^2(x_3+w),\Courant{x_1+u,x_2+v,x_4+t}]\\
 &\quad =[\Tilde{\alpha}^2(x_3+w),\Courant{x_1,x_2,x_4}+D_{\rho,\theta}(x_1,x_2)t+\theta(x_2,x_4)u-\theta(x_1,x_4)v+\huaG(x_1,x_2,x_4)]\\
 &\quad =[\alpha^2(x_3),\Courant{x_1,x_2,x_4}]+\rho(\alpha^2(x_3))D_{\rho,\theta}(x_1,x_2)t+\rho(\alpha^2(x_3))\theta(x_2,x_4)u-\rho(\alpha^2(x_3))\theta(x_1,x_4)v\\ 
 &\quad\quad +\rho(\alpha^2(x_3))\huaG(x_1,x_2,x_4)+\rho(\Courant{x_1,x_2,x_4})\beta^2(w)
 +\huaF(\alpha^2(x_3),\Courant{x_1,x_2,x_4})
\end{align*}
Then we have 
\begin{multline*}
\Courant{\Tilde{\alpha}(x_1+u),\Tilde{\alpha}(x_2+v),[x_3+w,x_4+t]}-[\Courant{x_1+u,x_2+v,x_3+w},\Tilde{\alpha}^2(x_4+t)] \\
-[\Tilde{\alpha}^2(x_3+w),\Courant{x_1+u,x_2+v,x_4+t}]\overset{\eqref{RL8}- \eqref{cocycle3}}{=}0.
\end{multline*}
Then $(A\oplus V ,[\cdot,\cdot]_\rho,\Courant{\cdot,\cdot,\cdot}_\theta,\alpha)$ is a Hom-Lie-Yamaguti algebra. The proof is completed.
\end{proof}
\begin{prop}
    A linear map $T: V\rightarrow A$  is an
    $(\huaF,\huaG)$-twisted $\mathcal{O}$-operator if and only if the graph
of T,
\begin{align*}
    Gr(T )=\{(T u,u)|\;u\in A\}
\end{align*} is a subalgebra of $(\huaF,\huaG)$-twisted semi-direct product  $A\ltimes_{\huaF,\huaG} V$.
\end{prop}
\begin{prop}\label{PP}
    Let $T: V\rightarrow A$ a
    $(\huaF,\huaG)$-twisted $\mathcal{O}$-operator. Then, $V$ carries a Hom-Lie-Yamaguti algebra structure with brackets defined for all $u, v, w \in V$ by 
\begin{align}
    \label{V-HLY1}
[u,v]_T&=\rho(Tu)v-\rho(Tv)u+\huaF(Tu,Tv),\\ 
\label{V-HLY2} \Courant{u,v,w}_T&=D_{\rho,\theta}(Tu,Tv)w+\theta(Tv,Tw)u-\theta(Tu,Tw)v+\huaG(Tu,Tv,Tw).
\end{align}
Moreover, $T$ is a morphism from $(V,[\cdot,\cdot]_T,\Courant{\cdot,\cdot,\cdot}_T,\beta)$ to $(A,[\cdot,\cdot],\Courant{\cdot,\cdot,\cdot},\alpha)$.
\end{prop}
\begin{proof}Let $u,v,w,a \in A$. On the one hand,
\begin{align*}
&\Courant{[u,v]_T,\beta(w),\beta(a)}_T=\Courant{\rho(Tu)v-\rho(Tv)u+\huaF(Tu,Tv),\beta(w),\beta(a)}_T\\
&\quad=D_{\rho,\theta}(T(\rho(Tu)v-\rho(Tv)u+\huaF(Tu,Tv)),T\beta(w))\beta(a)+\theta(T\beta(w),T\beta(a))\rho(Tu)v\\
&\quad\quad-\theta(T\beta(w),T\beta(a))\rho(Tv)u+\theta(T\beta(w),T\beta(a))\huaF(Tu,Tv)\\
&\quad\quad-\theta(T(\rho(Tu)v-\rho(Tv)u+\huaF(Tu,Tv)),T\beta(a))\beta(w)\\
&\quad\quad+\huaG(T(\rho(Tu)v-\rho(Tv)u+\huaF(Tu,Tv)),T\beta(w),T\beta(a))\\
&\quad=D_{\rho,\theta}([Tu,Tv],T\beta(w))\beta(a)+\theta(T\beta(w),T\beta(a))\rho(Tu)v-\theta(T\beta(w),T\beta(a))\rho(Tv)u\\
&\quad\quad+\theta(T\beta(w),T\beta(a))\huaF(Tu,Tv)-\theta([Tu,Tv],T\beta(a))\beta(w)+\huaG([Tu,Tv],T\beta(w),T\beta(a))\\
&\quad=D_{\rho,\theta}([Tu,Tv],\alpha T(w))\beta(a)+\theta(\alpha T(w),\alpha T(a))\rho(Tu)v-\theta(\alpha T(w),\alpha T(a))\rho(Tv)u\\
&\quad\quad+\theta(\alpha T(w),\alpha T(a))\huaF(Tu,Tv)-\theta([Tu,Tv],\alpha T(a))\beta(w)+\huaG([Tu,Tv],\alpha T(w),\alpha T(a)).
\end{align*}
 On the other hand, 
 \begin{align*}
&\Courant{[v,w]_T,\beta(u),\beta(a)}_T=D_{\rho,\theta}([Tv,Tw],T\beta(u))\beta(a)+\theta(T\beta(u),T\beta(a))\rho(Tv)w-\theta(T\beta(u),T\beta(a))\rho(Tw)v\\
&\quad\quad +\theta(T\beta(u),T\beta(a))\huaF(Tv,Tw)-\theta([Tv,Tw],T\beta(a))\beta(u)+\huaG([Tv,Tw],T\beta(u),T\beta(a)\\
&\quad=D_{\rho,\theta}([Tv,Tw],\alpha T(u))\beta(a)+\theta(\alpha T(u),\alpha T(a))\rho(Tv)w-\theta(\alpha T(u),\alpha T(a))\rho(Tw)v\\
&\quad\quad+\theta(\alpha T(u),\alpha T(a))\huaF(Tv,Tw)-\theta([Tv,Tw],\alpha T(a))\beta(u)+\huaG([Tv,Tw],\alpha T(u),\alpha T(a)),
 \\
    & \Courant{[w,u]_T,\beta(v),\beta(a)}_T=D_{\rho,\theta}([Tw,Tu],T\beta(v))\beta(a)+\theta(T\beta(v),T\beta(a))\rho(Tw)u-\theta(T\beta(v),T\beta(a))\rho(Tu)w\\
    &\quad\quad+\theta(T\beta(v),T\beta(a))\huaF(Tw,Tu)-\theta([Tw,Tu],T\beta(a))\beta(v)+\huaG([Tw,Tu],T\beta(v),T\beta(a))\\
    &\quad =D_{\rho,\theta}([Tw,Tu],\alpha T(v))\beta(a)+\theta(\alpha T(v),\alpha T(a))\rho(Tw)u-\theta(\alpha T(v),\alpha T(a))\rho(Tu)w\\
    &\quad\quad+\theta(\alpha T(v),\alpha T(a))\huaF(Tw,Tu)-\theta([Tw,Tu],\alpha T(a))\beta(v)+\huaG([Tw,Tu],\alpha T(v),\alpha T(a)).
 \end{align*}
 Then, 
 $\Courant{[u,v]_T,\beta(w),\beta(a)}_T+\Courant{[v,w]_T,\beta(u),\beta(a)}_T+\Courant{[w,u]_T,\beta(v),\beta(a)}_T=0.$
 
 Now, let us check the identity \eqref{LY3}. For $u,v,w,a \in V$, 
 \begin{align*}
  &   \Courant{\beta(u),\beta(v),[w,a]_T}_T=
  \Courant{\beta(u),\beta(v),\rho(Tw)a-\rho(Ta)w+\huaF(Tw,Ta)}_T\\
  &\quad=D_{\rho,\theta}(T\beta(u),T\beta(v))\rho(Tw)a-D_{\rho,\theta}(T\beta(u),T\beta(v))\rho(Ta)w\\
  &\quad\quad+D_{\rho,\theta}(T\beta(u),T\beta(v))\huaF(Tw,Ta)+\theta(T\beta(v),[Tw,Ta])\beta(u)\\
  &\quad\quad-\theta(T\beta(u),[Tw,Ta]\beta(v)+\huaG(T\beta(u),T\beta(v),[Tw,Ta])\\
  &\quad=D_{\rho,\theta}(\alpha T(u),\alpha T(v))\rho(Tw)a-D_{\rho,\theta}(\alpha 
T(u),\alpha T(v))\rho(Ta)w\\
&\quad\quad+D_{\rho,\theta}(\alpha T(u),\alpha T(v))\huaF(Tw,Ta)+\theta(\alpha T(v),[Tw,Ta])\beta(u)\\
&\quad\quad-\theta(\alpha T(u),[Tw,Ta]\beta(v)+\huaG(\alpha T(u),\alpha T(v),[Tw,Ta])
 \\
    & [\Courant{u,v,w}_T,\beta^2(a)]_T=[D_{\rho,\theta}(Tu,Tv)w+\theta(Tv,Tw)u-\theta(Tu,Tw)v+\huaG(Tu,Tv,Tw),\beta^2(a)]_T\\
    &\quad=\rho(\Courant{Tu,Tv,Tw})\beta^2(a)-\rho(T\beta^2(a))D_{\rho,\theta}(Tu,Tv)w-\rho(T\beta^2(a))\theta(Tv,Tw)u\\
    &\quad\quad+\rho(T\beta^2(a))\theta(Tu,Tw)v-\rho(T\beta^2(a))\huaG(Tu,Tv,Tw)+\huaF(\Courant{Tu,Tv,Tw},T\beta^2(a))\\
    &\quad =\rho(\Courant{Tu,Tv,Tw})\beta^2(a)-\rho(\alpha^2 T(a))D_{\rho,\theta}(Tu,Tv)w-\rho(\alpha^2 T(a))\theta(Tv,Tw)u\\
    &\quad \quad +\rho(\alpha^2 T(a))\mu(Tu,Tw)v-\rho(\alpha^2 T(a))\huaG(Tu,Tv,Tw)+\huaF(\Courant{Tu,Tv,Tw},\alpha^2 T(a)).
 \end{align*}
 Similarly, we have 
 \begin{align*}
   &  [\beta^2(w),\Courant{u,v,a}_T]_T=[\beta^2(w),D_{\rho,\theta}(Tu,Tv)a+\theta(Tv,Ta)u-\theta(Tu,Ta)v+\huaG(Tu,Tv,Ta)]_T\\
   &\quad=\rho(T\beta^2(w))D_{\rho,\theta}(Tu,Tv)a+\rho(T\beta^2(w))\theta(Tv,Ta)u-\rho(T\beta^2(w))\theta(Tu,Ta)v\\
   &\quad\quad+\rho(T\beta^2(w))\huaG(Tu,Tv,Ta)-\rho(\Courant{Tu,Tv,Ta})\beta^2(w)+\huaF(T\beta^2 (w),\Courant{Tu,Tv,Ta})\\
   &\quad=\rho(\alpha^2 T(w))D_{\rho,\theta}(Tu,Tv)a+\rho(\alpha^2 T(w))\theta(Tv,Ta)u-\rho(\alpha^2 T(w))\theta(Tu,Ta)v\\
   &\quad\quad+\rho(\alpha^2 T(w))\huaG(Tu,Tv,Ta)-\rho(\Courant{Tu,Tv,Ta})\beta^2(w)+\huaF(\alpha^2 T(w),\Courant{Tu,Tv,Ta}).
 \end{align*}

Similarly, we can prove the other identities.
Then we obtain that $(V,[\cdot,\cdot]_T,\Courant{\cdot,\cdot,\cdot}_T,\beta)$ is a Hom-Lie-Yamaguti algebra.  Moreover,
it is easy to show that $T$ is a morphism from $(V,[\cdot,\cdot]_T,\Courant{\cdot,\cdot,\cdot}_T,\beta)$ to $(A,[\cdot,\cdot],\Courant{\cdot,\cdot,\cdot},\alpha)$.
The proof is complete.
\end{proof}

    

\section{Cohomology of twisted $\mathcal O$-operator on a Hom-Lie-Yamaguti algebra}\label{Sec4}
In this section, we give a cohomology of twisted $\mathcal O$-operator on Hom-Lie-Yamaguti algebras. Let $T: V\rightarrow A$ a
    $(\huaF,\huaG)$-twisted $\mathcal O$-operator. We have seen in Proposition \ref{PP} that the linear
space $V$ carries a Hom-Lie-Yamaguti algebra structure with brackets
$(
[u,v]_T,\Courant{u,v,w}_T)$. We show that the cohomology of $T$ can be interpreted as the 
cohomology of $(V,[\cdot,\cdot]_T,\Courant{\cdot,\cdot,\cdot}_T,\beta)$ with coefficients in a suitable module structure on $A$. Define
the linear maps $\rho_T: V \to End(A)$ and $\theta_T : \otimes^2 V \to End(A)$ by  
\begin{align}
    \label{inducedRep}
&\rho_T(u)x= [Tu,x]+T(\rho(x)u+\huaF(x,Tu)),\\
   &\theta_T(u,v)x = [x,Tu,Tv]- T \Big( D(x,Tu)v- \theta(x,Tv)u +\huaG(x,Tu,Tv) \Big) \quad \forall u,v\in M, ~   ~ x\in L.
\end{align}
\begin{prop} With the above  notations, $(A,\rho_T,\theta_T,\alpha)$ is a representation  of the  Hom-Lie-Yamaguti algebra
$(V,[\cdot,\cdot]_T,[\cdot, \cdot, \cdot]_T,\beta).$
\end{prop}
\begin{proof}
Firstly, for all $ u,v\in M$ and $x\in L$, 
$$D_T(u,v)x=  \theta_T(v,u)x- \theta_T(u,v)x 
=[Tu,Tv,x]- T(\theta(Tv,x))u -\theta(Tu,x))v+ \huaG(Tu,Tv,x)) $$
By \eqref{cocycle1} and \eqref{R1}-\eqref{R3},
for all $u, v, w, \in V, x \in A$,  
\begin{align*}
&D_T(u,v)x-\theta_T(v,u)x+\theta_T(u,v)x+\rho_T([u,v])\circ\alpha( x)-\rho_T( \beta(u))\rho (v)x+\rho_T (\beta(v))\rho_T (u)x\\
& =[Tu,Tv,x]- T(\theta(Tv,x))u -\theta(Tu,x))v+ \huaG(Tu,Tv,x)) - [x,Tv,Tu]\\ 
&\quad + T \Big( D(x,Tv)u- \theta(x,Tu)v +\huaG(x,Tv,Tu) \Big)+[x,Tu,Tv]\\
&\quad- T \Big( D(x,Tu)v- \theta(x,Tv)u +\huaG(x,Tu,Tv)\Big)+[T[u,v],\alpha(x)]\\ 
&\quad+T(\rho(\alpha(x)[u,v]+\huaF(\alpha(x),T[u,v])\\
&\quad-[T(\beta(u),[Tv,x]]-T(\rho([Tv,x])\beta(u)-\huaF([Tv,x],Tu))-[T\beta(u),T\rho(x)v]-T(\rho(T\rho(x)v))\beta(u)\\
&\quad-\huaF(T\rho(x)v,T\beta(u))-[T\beta(u),T\huaF(x,Tv)]-T(\rho(T\huaF(x,Tv))\beta(u)-\huaF(\huaF(x,Tv),T\beta(u))\\
&\quad+[T\beta(v),[Tu,x]]+T(\rho([Tu,x])\beta(v)+\huaF([Tu,x],Tv)+[T\beta(v),T\rho(x)u]+T(\rho(T\rho(x)u))\beta(v)\\
&\quad+\huaF(T\rho(x)u,T\beta(v))+[T\beta(v),T\huaF(x,Tu)]+T(\rho(T\huaF(x,Tu))\beta(v)+\huaF(\huaF(x,Tu),T\beta(v))=0,
\end{align*}
On the other hand,  according to \eqref{R1}- \eqref{R2}- \eqref{R3}- \eqref{cocycle2}-\eqref{cocycle3}, 
\begin{align*}
    &D_T([u,v],\beta(w))x+D_T([v,w],\beta(u))x+D_T([w,u],\beta(v))x\\
    &=[T[u,v],T\beta(w),x]-T(\theta(T\beta(w),x)[u,v]-\theta(T[u,v],x)\beta(w)+\huaG(T[u,v],T\beta(w),x))\\
    &+[T[v,w],T\beta(u),x]-T(\theta(T\beta(u),x)[v,w]-\theta(T[v,w],x)\beta(u)+\huaG(T[v,w],T\beta(u),x))\\
    &+[T[w,u],T\beta(v),x]-T(\theta(T\beta(v),x)[w,u]-\theta(T[w,u],x)\beta(v)+\huaG(T[w,u],T\beta(v),x))=0.
\end{align*}
Similarly, 
\begin{align*}
    &\theta_T([u,v],\beta(w))\circ \alpha(x)-\theta_T(\beta(u),\beta(w))\rho_T(v)x+\theta_T(\beta(v),\beta(w))\rho_T(u)x\\
    &=[\alpha(x),T([u,v]),\beta(w)]- T \Big( D(\alpha(x),T([u,v])\beta(w)- \theta(\alpha(x),T\beta(w))[u,v] +\huaG(\alpha(x),T([u,v]),T\beta(w)) \Big)\\
    &-[\rho_T(v)x,T\beta(u),T\beta(w)]+T \Big( D(\rho_T(v)x,\beta(u))\beta(w)-\theta(\rho_T(v)x,T\beta(w))\beta(u)+\huaG(\rho_T(v)x,T\beta(u),T\beta(w)) \Big)\\
    &+[\rho_T(u)x,T\beta(v),T\beta(w)]-T \Big( D(\rho_T(u)x,\beta(v))\beta(w)-\theta(\rho_T(u)x,T\beta(w))\beta(v)+\huaG(\rho_T(u)x,T\beta(v),T\beta(w)) \Big)\\
    &=[\alpha(x),T([u,v]),\beta(w)]- T \Big( D(\alpha(x),T([u,v])\beta(w)- \theta(\alpha(x),T\beta(w))[u,v] +\huaG(\alpha(x),T([u,v]),T\beta(w)) \Big)\\
    &-[[Tv,x],T\beta(u),T\beta(w)]-[T(\rho(x)v),T\beta(u),T\beta(w)]-[T\huaF(x,Tv)),T\beta(u),T\beta(w)]\\
    &+T \Big( D([Tv,x],\beta(u))\beta(w)+T \Big( D(T(\rho(x)v),\beta(u))\beta(w)+T \Big( D(T\huaF(x,Tv)),\beta(u))\beta(w)\\
    &-\theta([Tv,x],T\beta(w))\beta(u)-\theta(T(\rho(x)v),T\beta(w))\beta(u)-\theta(T\huaF(x,Tv)),T\beta(w))\beta(u)\\
    &+\huaG([Tv,x],T\beta(u),T\beta(w)) \Big)
    +\huaG(T(\rho(x)v),T\beta(u),T\beta(w)) \Big)+\huaG(T\huaF(x,Tv)),T\beta(u),T\beta(w)) \Big)\\
    &+[[Tu,x],T\beta(v),T\beta(w)]+[T(\rho(x)u),T\beta(v),T\beta(w)]+[T\huaF(x,Tu)),T\beta(v),T\beta(w)]\\
    &-T ( D([Tu,x],\beta(v))\beta(w)-T ( D(T(\rho(x)u),\beta(v))\beta(w)-T ( D(T(\huaF(x,Tu)),\beta(v))\beta(w)\\
    &-\theta([Tu,x],T\beta(w))\beta(v)-\theta(T(\rho(x)u),T\beta(w))\beta(v)-\theta(T(\huaF(x,Tu)),T\beta(w))\beta(v)\\
    &+\huaG([Tu,x],T\beta(v),T\beta(w)) +\huaG(T(\rho(x)u),T\beta(v),T\beta(w))+\huaG(T\huaF(x,Tu)),T\beta(v),T\beta(w))=0.
\end{align*}
Moreover,
\begin{align*}
&D_T(\beta(u),\beta(v))\rho_T(w)(x)-\rho_T(\beta^2(w))D_T(u,v)(x)-\rho_T([u, v, w])\circ \alpha(x)\\ 
& =[T\beta(u),T\beta(v),\rho_T(w)(x)]- T\Big(\theta(T\beta(v),\rho_T(w)(x))\beta(u) -\theta(T\beta(u),\rho_T(w)(x))\beta(v)\\ 
&\hspace{9cm} + \huaG(T\beta(u),T\beta(v),\rho_T(w)(x))\Big)\\
&\quad-[T\beta^2(w),D_T(u,v)(x)]-T(\rho(D_T(u,v)(x))\beta^2(w)+\huaF(D_T(u,v)(x),T\beta^2(w)))\\
&\quad-[T[u, v, w],\alpha(x)]-T(\rho(\alpha(x))[u, v, w]+\huaF(\alpha(x),T[u, v, w])))\\
=&[T\beta(u),T\beta(v),[Tw,x]]+[T\beta(u),T\beta(v),T(\rho(x)w)]+[T\beta(u),T\beta(v),T\huaF(x,Tw))]\\
&- T\Big(\theta(T\beta(v),[Tu,x]))\beta(u)+ \theta(T\beta(v),T(\rho(x)u))\beta(u)+ \theta(T\beta(v),T\huaF(x,Tu))\beta(u)\Big)\\&-\theta(T\beta(u),[Tw,x])\beta(v)-\theta(T\beta(u),T(\rho(x)w))\beta(v)-\theta(T\beta(u),T\huaF(x,Tw))\beta(v)\\&+ \huaG(T\beta(u),T\beta(v),[Tw,x])+ \huaG(T\beta(u),T\beta(v),T(\rho(x)w))+ \huaG(T\beta(u),T\beta(v),T\huaF(x,Tw))\\&-[T\beta^2(w),[Tu,Tv,x]]+[T\beta^2(w),T(\theta(Tv,x))u )] \\
&+[T\beta^2(w),T(\theta(Tu,x))v)]-[T\beta^2(w),T \huaG(Tu,Tv,x))]\\
&-T\Big(\rho([Tu,Tv,x])\beta^2(w)-\rho(T(\theta(Tv,x))u )\beta^2(w)-\rho( T(\theta(Tu,x))v)\beta^2(w)+\rho(T \huaG(Tu,Tv,x))\beta^2(w)\Big)\\
&+\huaF([Tu,Tv,x],T\beta^2(w))-\huaF( T(\theta(Tv,x))u),T\beta^2(w))\\
&-\huaF(T(\theta(Tu,x))v),T\beta^2(w))+\huaF( T \huaG(Tu,Tv,x),T\beta^2(w)))\\
&-[T[u, v, w],\alpha(x)]-T(\rho(\alpha(x))[u, v, w]+\huaF(\alpha(x),T[u, v, w]))=0.
\end{align*}
Similarly, 
\begin{align*}
&\theta_T(\beta(u),[v, w])\circ \alpha(x)-\rho_T(\beta^2(v))\theta_T(u, w)x+\rho_T(\beta^2(w))\theta_T(u,v)x=0,\\
&D_T(\beta^2(u),\beta^2(v)\theta_T(w,a)=\theta_T(\beta^2(w),\beta^2(a))D_T(u,v)
+\theta_T([u,v,w],\beta^2(a))\circ\alpha^2(x)\\
&\hspace{8cm} +\theta_T(\beta^2(w),[u,v,a])\circ\alpha^2(x),\\
&\theta_T(\beta^2(u),[v, w, a])\circ\alpha^2(x)=\theta_T (\beta^2(v), \beta^2(a))\theta_T(u,v)
-\theta_T (\beta^2(v), \beta^2(a))\theta_T(u,w) \\
&\hspace{8cm} + D_T (\beta^2(v), \beta^2(w))\theta_T(u,a).
\end{align*}
Thus $(A,\rho_T,\theta_T,\alpha)$ is a representation of the Hom-Lie-Yamaguti algebra
$(V,[\cdot,\cdot]_T,[\cdot, \cdot, \cdot]_T,\beta).$
\end{proof}
Next, we give the coboundary operator of the Hom-Lie-Yamaguti algebra $(V, [\cdot, \cdot]_T, \Courant{\cdot, \cdot, \cdot}_T,\beta)$ with coefficients in the representation
$(A; \rho_T, \theta_T,\alpha)$. More precisely,  $\delta^T: (\huaF,\huaG)\mapsto (\delta^T_{\textrm{I}}\huaF, \delta^T_{\textrm{II}}\huaG)$  is given by
\begin{align*}
 &(\delta^T_{\textrm{I}}\huaF)(u_{1}, u_{2}, \cdots, u_{2n+2})\\
&=[T\beta^{2n}(u_{2n+1}),\huaG(u_{1}, \cdots, u_{2n}, u_{2n+2})]+T\rho(\beta^{2n}(u_{2n+1}))\huaG(u_{1}, \cdots, u_{2n}, u_{2n+2})\\
&+T\huaF(\huaG(u_{1}, \cdots, u_{2n}, u_{2n+2}),T(\beta^{2n}(u_{2n+1}))-[T\beta^{2n}(u_{2n+2}),\huaG(u_{1}, \cdots, u_{2n+1})]\\
&-T(\rho(\beta^{2n}(u_{2n+2}))\huaG(u_{1}, \cdots, u_{2n+1})
-T\huaF(\huaG(u_{1}, \cdots, u_{2n+1}),T\beta^{2n}(u_{2n+2}))\\
&-\huaG(\beta(u_{1}), \cdots, \beta(u_{2n}), \rho(Tu_{2n+1})u_{2n+2}-\rho(Tu_{2n+2})u_{2n+1}+\huaF(Tu_{2n+1},Tu_{2n+2}))\\
&+\sum\limits_{k=1}^{n}(-1)^{n+k+1}[T\beta^{2n-1}(u_{2k-1}), T\beta^{2n-1}(u_{2k})),\huaF(u_{1}, \cdots, \hat{u}_{2k-1}, \hat{u}_{2k}, \cdots, u_{2n+2})]\\&-T\theta(T\beta^{2n-1}(u_{2k})),\huaF(u_{1}, \cdots, \hat{u}_{2k-1}, \hat{u}_{2k}, \cdots, u_{2n+2}))\beta^{2n-1}(u_{2k-1})\\&-T\theta(T\beta^{2n-1}(u_{2k-1})),\huaF(u_{1}, \cdots, \hat{u}_{2k-1}, \hat{u}_{2k}, \cdots, u_{2n+2}))\beta^{2n-1}(u_{2k})\\&+T\huaG(T\beta^{2n-1}(u_{2k-1}), T\beta^{2n-1}(u_{2k})),T\huaF(u_{1}, \cdots, \hat{u}_{2k-1}, \hat{u}_{2k}, \cdots, u_{2n+2}))\\&+\sum\limits_{k=1}^{n}\sum\limits_{j=2k+1}^{2n+2}(-1)^{n+k}\huaF(\beta^2(u_{1}), \cdots, \hat{u}_{2k-1}, \hat{u}_{2k}, \cdots,D(Tu_{2k-1},Tu_{2k})u_{j}\\
&\hspace{2cm} 
+\theta(Tu_{2k},Tu_{j})u_{2k-1}-\theta(Tu_{2k-1},Tu_{j})u_{2k}
+\huaG(Tu_{2k-1},Tu_{2k},Tu_{j}), \cdots, \beta^2(u_{2n+2})).\\
   &(\delta_{\textrm{II}}\huaG)(u_{1}, u_{2}, \cdots, u_{2n+3})\\
&=\Courant{\huaG(x_{1},\cdots, u_{2n+1}),T\beta^{2n}(u_{2n+2}), T\beta^{2n}(u_{2n+3})}-T(D(\huaG(x_{1},\cdots, u_{2n+1}),T\beta^{2n}(u_{2n+2}))\beta^{2n}(u_{2n+3})\\
&-\theta(\huaG(x_{1},\cdots, u_{2n+1}),T\beta^{2n}(u_{2n+3}))\beta^{2n}(u_{2n+2})+\huaG(\huaG(x_{1},\cdots, u_{2n+1}),T\beta^{2n}(u_{2n+2}), T\beta^{2n}(u_{2n+3}))\\
&-\Courant{\huaG(u_{1}, \cdots, u_{2n}, u_{2n+2}),T\beta^{2n}(u_{2n+1}),T \beta^{2n}(u_{2n+3}))}\\
&+T(D(\huaG(u_{1}, \cdots, u_{2n}, u_{2n+2}),T\beta^{2n}(u_{2n+1})\beta^{2n}(u_{2n+3})\\
&-\theta(\huaG(u_{1}, \cdots, u_{2n}, u_{2n+2}),T\beta^{2n}(u_{2n+3}))\beta^{2n}(u_{2n+1})\\ &+\huaG(T\beta^{2n}(u_{2n+1}),T \beta^{2n}(u_{2n+3}),\huaG(u_{1}, \cdots, u_{2n}, u_{2n+2}))\\
&+\sum\limits_{k=1}^{n+1}(-1)^{n+k+1}\Courant{T\beta^{2n}(u_{2k-1}), T\beta^{2n}(u_{2k}),\huaG(u_{1}, \cdots, \hat{u}_{2k-1}, \hat{u}_{2k}, \cdots, u_{2n+3})}\\&-T\Big(\theta( T\beta^{2n}(u_{2k}),\huaG(u_{1}, \cdots, \hat{u}_{2k-1}, \hat{u}_{2k}, \cdots, u_{2n+3}))\beta^{2n}(u_{2k-1})\\&-\theta(T\beta^{2n}(u_{2k-1}),\huaG(u_{1}, \cdots, \hat{u}_{2k-1}, \hat{u}_{2k}, \cdots, u_{2n+3}))\beta^{2n}(u_{2k})\\&+\huaG(T\beta^{2n}(u_{2k-1}), T\beta^{2n}(u_{2k}),\huaG(u_{1}, \cdots, \hat{u}_{2k-1}, \hat{u}_{2k}, \cdots, u_{2n+3})\Big)\\
&+\sum\limits_{k=1}^{n+1}\sum\limits_{j=2k+1}^{2n+3}(-1)^{n+k}\huaG(\beta^2(u_{1}), \cdots, \hat{u}_{2k-1}, \hat{u}_{2k}, \cdots, \Courant{u_{2k-1}, u_{2k}, u_{j}}, \cdots, \beta^2(u_{2n+3})).\\
&+\sum\limits_{k=1}^{n+1}\sum\limits_{j=2k+1}^{2n+3}(-1)^{n+k}\huaG(\beta^2(u_{1}), \cdots, 
\hat{u}_{2k-1}, \hat{u}_{2k}, \cdots, D(Tu_{2k-1}, Tu_{2k}) u_{j}\\
&\hspace{7cm} +\theta(Tu_{2k}, Tu_{j})u_{2k-1}, \cdots, \beta^2(u_{2n+3})).
\end{align*}
With this coboundary the Hom-Lie-Yamaguti cochain forms a complex
$$ \mathfrak{C}^{1}(V,A)\overset{\delta_T^{1}}{\longrightarrow} \mathfrak{C}^{3}(V,A)\overset{\delta_T^{3}}{\longrightarrow}  \mathfrak{C}^{5}(V,A) \longrightarrow \cdots,$$
such that $\delta_T^{2n+1}\circ \delta_T^{2n-1}=0$ for any  $n \geq 1 $.
For any $\chi =(x_1,x_2) \in A  \times A $, we define $\partial_T(\chi) :V \rightarrow A$ by :
 \begin{equation}
    \partial_T(\chi)u=T\Big(D(\chi)u + \huaG(\chi,Tu)\Big)-[\chi,Tu]  \quad \forall  u\in V.
 \end{equation}
\begin{prop}\label{1-cocycle}
 Let $T$ be a twisted $\mathcal{O}$-operator on a 
Hom-Lie-Yamaguti algebra $(A,[\cdot,\cdot],\Courant{\cdot,\cdot,\cdot},\alpha)$ with respect to the representation $(V,\rho,\theta,\beta)$. Then $\partial_T(\chi)$ is a $1$-cocycle on the Hom-Lie-Yamaguti  algebra $(V,[\cdot,\cdot]_T,\Courant{\cdot,\cdot,\cdot}_T,\beta)$ with coefficients in $(A,\rho_T,\theta_T,\alpha)$.
\end{prop}
Let us describe cohomology of twisted $\mathcal{O}$-operator on
Hom-Lie-Yamaguti algebra $(A,[\cdot,\cdot],\Courant{\cdot,\cdot,\cdot},\alpha)$ with respect to the representation $(V,\rho,\theta,\beta)$. 
Let 
\begin{equation*}
  {\mathcal{Z}^n_T(V,A)}=\{f \in C^n_T(V,A)|d_T(f)=0\}, \quad 
  {\mathcal{B}^n_T(V,A)}=\{d_T(g)| g \in C^{n-1}_T(V,A)\}.
\end{equation*}
\begin{defn}
 Let $T$ be $(\huaF,\huaG)$-twisted $\mathcal{O}$-operator on a Hom-Lie-Yamaguti algebra $(A, [\cdot, \cdot], \Courant{\cdot,\cdot,\cdot},\alpha)$ 
with respect to the representation $(V, \rho, \theta,\beta)$.
 Denote the set of cocycles by ${\mathcal{Z}}^{\ast}(V,A)$, the set of coboundaries by ${\mathcal{B}}^{\ast}(V,A)$
 and the cohomology group by
 $${\mathbf{H}}_T^{\ast}(V,A)={\mathcal{Z}}^{\ast}(V,A) / {\mathcal{B}}^{\ast}(V,A).$$
\end{defn}
\section{Deformations of twisted $\mathcal O$-operator on a Hom-Lie-Yamaguti algebra}\label{Sec5}
In this section, we study  deformations  of twisted $\mathcal{O}$-operator on a Hom-Lie-Yamaguti algebra via the cohomology theory established in the former
section.
 Let $\mathbb{K}[[t]]$ be the ring of power series in one variable $t$. For any $\mathbb{K}$-linear space $V$, we denote by  $M[[t]]$ the
linear space of formal power series in $t$ with coefficients in $M$. If in addition, we have a structure of Hom-Lie-Yamaguti algebra $(A, [\cdot,\cdot],\Courant{\cdot,\cdot,\cdot},\alpha)$ over $\mathbb{K}$, then there is a  Hom-Lie-Yamaguti algebra structure over the ring $\mathbb{K}[[t]]$ on $A[[t]]$ given
by
\begin{equation}\label{S1}
\Big[\sum_{i=0}^{+\infty}t^{i}x_i,\sum_{j=0}^{+\infty}t^{j}y_j\Big\}=\sum_{s=0}^{+\infty}\sum_{i+j=s}t^{s}[x_i,y_j],
\end{equation}

\begin{equation}\label{S2}
\Big\{\sum_{i=0}^{+\infty}t^{i}x_i,\sum_{j=0}^{+\infty}t^{j}y_j,\sum_{k=0}^{+\infty}t^{k}z_k\Big\}=\sum_{s=0}^{+\infty}\sum_{i+j+k=s}t^{s}\{x_i,y_j,z_k\},\;\forall x_i,y_j,z_k\in A.
\end{equation}
For any representation
$(V,\rho,\theta,\beta)$ of a Hom-Lie-Yamaguti algebra $(A, [\cdot,\cdot],\Courant{\cdot,\cdot,\cdot},\alpha)$, there is a
nature representation of the Hom-Lie-Yamaguti algebra
$A[[t]]$ on the $\mathbb{K}[[t]]$-module $V[[t]]$, which is given by
\begin{equation}
    \rho\Big(\sum_{i=0}^{+\infty}t^{i}x_i\Big) \Big(\sum_{j=0}^{+\infty}t^{j}u_j\Big)=\sum_{s=0}^{+\infty}\sum_{i+j=s}t^{s}\rho(x_i)u_j ,
\end{equation}
\begin{equation}
\theta\Big(\sum_{i=0}^{+\infty}t^{i}x_i,\sum_{j=0}^{+\infty}t^{j}y_j\Big)\Big(\sum_{k=0}^{+\infty}t^{k}u_k\Big)=\sum_{s=0}^{+\infty}\sum_{i+j+k=s}t^{s}\theta(x_i,y_j)u_k,
\end{equation}
\begin{equation}
D\Big(\sum_{i=0}^{+\infty}t^{i}x_i,\sum_{j=0}^{+\infty}t^{j}y_j\Big)\Big(\sum_{k=0}^{+\infty}t^{k}u_k\Big)=\sum_{s=0}^{+\infty}\sum_{i+j+k=s}t^{s}D(x_i,y_j)u_k,\;\forall x_i,y_j\in A,\;u_k\in V.
\end{equation}
Similarly, the $(2,3)$-cocycle $(\huaF,\huaG)$ can be extended to a $(2,3)$-cocycle ( denoted by the same notation $(\huaF,\huaG)$) on the Hom-Lie-Yamaguti structure $A[[t]]$ with coefficients in $V[[t]]$. Consider a power series
\begin{equation}\label{T1}
T_t=\sum_{i=0}^{+\infty}t^{i}T_i,\quad\;T_i\in  Hom(V,A).
\end{equation}
That is $T_t\in  Hom(V,A[[t]])$. Extend it to be a $\mathbb{K}[[t]]$-module map from
$V[[t]]$ to $A[[t]]$ which is still denoted by $T_t$.
\begin{defn}
If the  power series $T_t=\displaystyle\sum_{i=0}^{+\infty}T_it^{i}$ with $T_0=T$ satisfies for all $u,v,w\in V$ the conditions  
\begin{align}
\beta T_t&=T_t\alpha,\\
\label{tt1}[T_t u,T_t v]&=T_t\Big(\rho(T_t u)v-\rho(T_t v)u+\huaF(T_t u,T_t v)\Big),\\ \label{tt2}
\{T_t u,T_t v,T_t w\}&=T_t\Big(D(T_t u,T_t v)w+\theta(T_t v,T_t w)u-\theta(T_t u,T_t w)v + \huaG(T_t u,T_t v,T_t w) \Big),
\end{align} 
then we say that it is a \textbf{formal deformation} of the  twisted $\mathcal O$-operator $T$.
\end{defn}
Using \eqref{S1}, \eqref{S2} and \eqref{T1}  to expand \eqref{tt1} and \eqref{tt2},  it becomes evident that  \eqref{tt1} and \eqref{tt2}  can be expressed as the following system of
equations $\forall u,v,w\in V$: 
\begin{align}
\label{tt}
\beta T_s&=T_s\alpha,\\
\sum_{i+j=s}[T_i u,T_j v]&=\sum_{i+j=s}T_i\Big(\rho(T_j u)v-\rho(T_j v)u\Big)+\sum_{i+j+k=s}T_i\Big(\huaF(T_j u,T_k v)\Big),\\ \nonumber
\sum_{i+j+k=s}\{T_i u,T_j v,T_k w\}&=\sum_{i+j+k=s}T_i\Big(D(T_j u,T_k v)w+\theta(T_j v,T_k w)u-\theta(T_j u,T_k w)v \Big) \nonumber\\&+\sum_{i+j+k+l=s} T_i\Big(\huaG(T_j u,T_k v,T_l w)\Big).  
\end{align}
 Observe that for $s=0$ and $T_0=T$,  \eqref{tt}-  \eqref{tt1} and \eqref{tt2}   give the
 definition of twisted $\mathcal O$-operator.
 \begin{prop}
 Let $T_t=\displaystyle\sum_{i=0}^{+\infty}T_it^{i}$ 
is a formal deformation of a twisted $\mathcal{O}$-operator 
on a Hom-Lie-Yamaguti algebra $(A, [\cdot,\cdot],\Courant{\cdot,\cdot,\cdot},\alpha)$ with respect to the representation $(V,\rho,\theta,\beta)$. Then,  
$T_1$
is a $1$-cocycle of the twisted $\mathcal{O}$-operator $T$, called the infinitesimal of the deformation $T_t$.   
 \end{prop}
\begin{proof} For $u,v,w\in V$,  
    \begin{align*}
        & \beta T_1=T_1\alpha,\\
        & [T_1 u,Tv]+[Tu,T_1 v]=T_1\big(\rho(Tu)v-\rho(Tv)u\big)+T\big(\rho(T_1u)v-\rho(T_1v)u\big)+T_1 \huaF(Tu,Tv)\\
        &+T\huaF(T_1 u,Tv)+T\huaF(Tu,T_1v),\\ 
        &\{T_1 u,Tv,Tw\}+\{T u,T_1v,Tw\}+\{T u,Tv,T_1w\}=T_1\Big(D(Tu,Tv)w+\theta(Tv,Tw)u-\theta(Tu,Tw)v\Big)\\
        &+T\Big(D(T_1u,Tv)w+\theta(T_1v,Tw)u-\theta(T_1u,Tw)v\Big)\\
        &+T\Big(D(Tu,T_1v)w+\theta(Tv,T_1w)u-\theta(Tu,T_1w)v\Big)\\
        &+T_1\big(\huaG(Tu,Tv,Tw)\big)+T\big(\huaG(T_1u,Tv,Tw)\big)\\
        &+T\big(\huaG(Tu,T_1v,Tw)\big)+T\big(\huaG(Tu,Tv,T_1w)\big).
    \end{align*} Direct deduction shows that $T_1 \in  C^1_{\rm HLY}(A,V) $
is a $1$-cocycle   in the cohomology of the twisted $\mathcal O$-operator T. Further, the $1$-cocycle $T_1$ is called the infinitesimal
of the formal deformation $T_t$ of $T$.
\end{proof}
 Let $T$ be a twisted $\mathcal O$-operator on a Hom-Lie-Yamaguti algebra $(A, [\cdot,\cdot],\Courant{\cdot,\cdot,\cdot},\alpha)$ with respect to the representation $(V,\rho,\theta,\beta)$, and   $T_t=\displaystyle\sum_{i=0}^{+\infty}T_it^{i}$ and $T^{'}_t=\displaystyle\sum_{i=0}^{+\infty}T^{'}_it^{i}$ 
are two formal deformations of a twisted $\mathcal{O}$-operator on $A$.
\begin{defn}
Two formal deformations $T_t$ and $T^{'}_t$ are called equivalent if there exists
an element $\chi \in A\wedge A$, $\phi_i\in End(A)$ and $\psi_i\in End(V)$ $(i\geq2)$, such that the couple 
\begin{align}
    \label{eq:equi1}
\phi_t=&Id_A+t[\chi,-]+\sum_{i=2}^{+\infty}\phi_it^{i},\\ \label{eq:equi2} \psi_t=&Id_V+t \Big(D(\chi)(-)+ \huaG(\chi,T-)\Big)+\sum_{i=2}^{+\infty}\psi_it^{i}.
\end{align}
is a morphism of   twisted $\mathcal O$-operators   from  $T_t$ to $T^{'}_t$ given in Definition \ref{Tmorphism}
\end{defn}
Let $T$ be a twisted $\mathcal O$-operator on a Hom-Lie-Yamaguti algebra  $(A,[\cdot,\cdot],\Courant{\cdot,\cdot,\cdot},
\alpha)$
with respect to the representation 
$(V,\rho,\theta,\beta)$.
\begin{thm}
    Consider two equivalent formal deformations $T_t=\displaystyle\sum_{i=0}^{+\infty}T_it^{i}$ and $T^{'}_t=\displaystyle\sum_{i=0}^{+\infty}T^{'}_it^{i}$ of $T$ , then $T_t$
and $T^{'}_t$ define the same cohomology class in ${\mathbf{H}}_T^{1}(V,A)$.
\end{thm}
\begin{proof}
    Let $\phi_t$ and $\psi_t$ are two linear maps defined in Equations \eqref{eq:equi1} and \eqref{eq:equi2} respectively  such that
two deformations  $T_t=\displaystyle\sum_{i=0}^{+\infty}T_it^{i}$ and $T^{'}_t=\displaystyle\sum_{i=0}^{+\infty}T^{'}_it^{i}$ are equivalent. Since we have $T_t(\phi_t(u))=\psi_t(T^{'}_t(u)), \forall u\in V$, then we have $ T^{'}_1(u)-T_1(u)=T\Big(D(\chi)u + \huaG(\chi,Tu)\Big)-[\chi,Tu]$. According to Proposition \ref{1-cocycle} it follows that $ T^{'}_1(u)-T_1(u)=\partial_T(\chi)u \in \mathcal{B}^1_T(V,A)$. Which implies that $T_t$
and $T^{'}_t$ define the same cohomology class in ${\mathbf{H}}_T^{1}(V,A)$.
\end{proof}
\section{NS-Hom-Lie-Yamaguti algebras}\label{Sec6}
In this Section, we introduce the notion of a Hom-NS-Lie-Yamaguti algebra. A Hom-NS-Lie-Yamaguti algebra gives
rise to a Hom-Lie-Yamaguti algebra and a representation on itself. We show that a twisted $\mathcal O$-operator on a Hom-Lie-Yamaguti algebra
induces an Hom-NS-Lie-Yamaguti algebra. Then, Hom-NS-Lie-Yamaguti algebra can be viewed as the underlying algebraic
structures of twisted $\mathcal{O}$-operator on Hom-Lie-Yamaguti algebra. Also, we show that an Hom-NS-Lie algebra give rise to
an Hom-NS-Lie-Yamaguti algebra.
\subsection{Definition and constructions}
\begin{defn}
An NS-Hom-Lie-Yamaguti algebra is a linear space $A$ together with two bylinear operations  $\curlyvee , \circ: A\times A\to A$ in which $\curlyvee$ is skew-symmetric and 
   $\{\cdot,\cdot,\cdot \}$,  $[\cdot,\cdot,\cdot]:A\times A\times A\to A$ are two multilinear ternary operations, which with respect to a linear operator $\alpha:A\to A$ satisfy the following identities for all $x_1,x_2,x_3,x_4\in A:$  
   \begin{equation}
        \label{NScocycle1}
\big([x_1,x_2]^*\curlyvee \alpha(x_3)-\alpha(x_1)\circ( x_2\curlyvee x_3)+[x,y,z]\big)=0,\end{equation}
  \begin{equation} \label{NScocycle2}\circlearrowleft_{x_1,x_2,x_3}\big(\{x_2\curlyvee x_3,\alpha(x_1),\alpha(x_4)\}+[[x_1,x_2]^*,\alpha(x_3),\alpha(x_4)]\big)=0,\end{equation}
  \begin{align} 
&\circlearrowleft_{x_1,x_2,x_3}\{[x_1,x_2]^*,\alpha(x_3),\alpha(x_4)\}=0,\end{align}
  \begin{align}
       &\{\alpha(x_4),[x_1,x_2]^*,\alpha(x_3)\}=\{x_2\circ x_4,\alpha(x_1),\alpha(x_3)\}-\{x_1\circ x_4,\alpha(x_2),\alpha(x_3)\},\end{align}
  \begin{align}\label{39}
       &\{\alpha(x_1),\alpha(x_2),x_3\circ x_4\}^{\wedge}=\alpha^2(x_3)\circ \{x_1,x_2,x_4\}^{\wedge}+\Courant{x_1,x_2,x_3}\circ\alpha^2(x_4),\end{align}
  \begin{align}\label{40}
       &\{\alpha(x_4),\alpha(x_1),[x_2,x_3]^*\}=\alpha^2(x_2)\circ\{x_4,x_1,x_3\}-\alpha^2(x_3)\circ \{x_4,x_1,x_2\},\end{align}
  \begin{align}\nonumber 
&\{\alpha^2(x_1),\alpha^2(x_2),\{x_1,x_2,x_4\}\}^{\wedge}\\
&=\{\{x_1,x_2,x_5\}^{\wedge},\alpha^2(x_3),\alpha^2(x_4)\}+\{\alpha^2(x_5),\Courant{x_1,x_2,x_3},\alpha^2(x_4)\}+\{\alpha^2(x_5),\alpha^2(x_3),\Courant{x_1,x_2,x_4}\},\end{align}
  \begin{align}
\nonumber & \{\alpha^2(x_5),\alpha^2(x_1),\Courant{x_2,x_3,x_4}\}\\
&=\{\{x_5,x_1,x_2\},\alpha^2(x_3),\alpha^2(x_4)\}-\{\{x_5,x_1,x_3\},\alpha^2(x_2),\alpha^2(x_4)\}+\{\alpha^2(x_2),\alpha^2(x_3),\{x_5,x_1,x_4\}\}^{\wedge},
  \end{align}
  \begin{align}    \label{NScocycle3}
&\{\alpha(x_1),\alpha(x_2),x_3\curlyvee x_4\}+\alpha^2(x_4)\circ [x_1,x_2,x_3]-\Courant{x_1,x_2,x_3}\curlyvee \alpha^2(x_4)\nonumber\\&\quad\quad\quad+[\alpha(x_1),\alpha(x_2),[x_3,x_4]^*]-\alpha^2(x_3)\circ [x_1,x_2,x_4]-\alpha^2(x_3)\curlyvee \Courant{x_1,x_2,x_4})=0,\end{align}
\begin{align} 
  &\nonumber \{[x_1,x_2,x_3],\alpha^2(x_4),\alpha^2(x_5)\}^\wedge-\{[x_1,x_2,x_4],\alpha^2(x_3),\alpha^2(x_5)\}+\{[x_3,x_4,x_5],\alpha^2(x_1),\alpha^2(x_2)\}^\wedge \\ \nonumber&-\{[x_1,x_2,x_5],\alpha^2(x_3),\alpha^2(x_4)\}^\wedge -[\Courant{x_1,x_2,x_3},\alpha^2(x_4),\alpha^2(x_5)]-
[\alpha^2(x_3),\Courant{x_1,x_2,x_4},\alpha^2(x_5)]\\&+[\alpha^2(x_1),\alpha^2(x_2),\Courant{x_3,x_4,x_5}]-[\alpha^2(x_3),\alpha^2(x_4),\Courant{x_1,x_2,x_5}]=0.\label{NScocycle4}
\end{align}
       where $[\cdot,\cdot]^*$, $\{\cdot,\cdot,\cdot\}^{\wedge}$ and   $\Courant{\cdot,\cdot,\cdot}$ are defined by
       \begin{align}
         &[x_1,x_2]^*=x_1\circ x_2-x_2\circ x_1 +x_1\curlyvee x_2,\\ 
         &\{x_1,x_2,x_3\}^{\wedge}=\{x_3,x_2,x_1\}-
         \{x_3,x_1,x_2\}+\alpha(x_1)\circ(x_2\circ x_3) \nonumber  \\
         &\hspace{5cm} -\alpha(x_2)\circ(x_1 \circ x_3)-[x_1,x_2]^*\circ \alpha(x_3),\\
         &\Courant{x_1,x_2,x_3}=\{x_1,x_2,x_3\}^{\wedge}+\{x_1,x_2,x_3\}-\{x_2,x_1,x_3\}+[ x_1,x_2,x_3].
         \end{align}
\end{defn}
\begin{re}
    If the trilinear map $\{\cdot, \cdot, \cdot\}$ and the bilinear map $\circ$ in the above definition are trivial, then $(A,\curlyvee,[\cdot,\cdot,\cdot],\alpha)$ becomes  a Hom-Lie-Yamaguti algebra. On the other hand, if  $[\cdot,\cdot,\cdot]$ and $\curlyvee$ are trivial, one gets the Hom version of pre-Lie-Yamaguti algebra given in \textnormal{\cite{Sheng1}}. Thus, NS-Hom-Lie-Yamaguti algebras
are a generalization of both Hom-Lie-Yamaguti algebras and Hom-pre-Lie-Yamaguti algebras. Moreover if $\{\cdot, \cdot, \cdot\}$, $[\cdot,\cdot]$ and $\circ$ are trivial, then $(A,[\cdot,\cdot,\cdot],\alpha)$ is a Hom-Lie triple system. On the other hand, if $[\cdot,\cdot,\cdot], \curlyvee$ and $\circ$ are trivial, one gets the Hom version of pre-Lie-triple system given in \textnormal{\cite{plts}}.
\end{re}
\begin{thm}
Under the above notations,$(A,[\cdot,\cdot]^*,\Courant{\cdot,\cdot,\cdot},\alpha)$ is a Hom-Lie-Yamaguti algebra, which is called the
sub-adjacent Hom-Lie-Yamaguti algebra of $(A,\curlyvee,\{\cdot,\cdot,\cdot\},\circ,\Courant{\cdot,\cdot,\cdot})$  and denoted by $\mathcal{A}$. Moreover, $(A,\rho,\theta)$  is a representation of $\mathcal{A}$, where the
linear map $\rho:A\rightarrow End(A)$ and the
bilinear map $\theta:A\times A\rightarrow End(A)$ are  defined for all $x,y,z\in A$ by 
\begin{align*}
    &\rho(x)y=x\circ y\quad \text{and} \quad\theta(x,y)z=\{z,x,y\}.
\end{align*}
\end{thm}
\begin{proof}
For $x,y,z,w \in A$, by \eqref{39}-\eqref{40}, we have 
\begin{align*}
   & \Courant{\alpha(x),\alpha(y),[z,w]^*}-[\Courant{x,y,z},\alpha^2(w)]^*-[\alpha^2(z),\Courant{x,y,w}]^*\\
   &=[\alpha(x),\alpha(y),[z,w]^*]+\{\alpha(x),\alpha(y),[z,w]^*\}^{\wedge}+\{\alpha(x),\alpha(y),[z,w]^*\}-\{\alpha(y),\alpha(x),[z,w]^*\}\\
   &\quad -\Courant{x,y,z}\circ \alpha^2(w)+\alpha^2(w)\circ \Courant{x,y,z}-\Courant{x,y,z}\curlyvee \alpha^2(w)\\
   &\quad -\alpha^2(z) \circ \Courant{x,y,w}+\Courant{x,y,w}\circ \alpha^2(z)-\alpha^2(z)\curlyvee \Courant{x,y,w}\\
   &=[\alpha(x),\alpha(y),[z,w]^*]+\{\alpha(x),\alpha(y),z\circ w\}^{\wedge}+\{\alpha(x),\alpha(y),w\circ z\}^{\wedge}+\{\alpha(x),\alpha(y),z\curlyvee w\}^{\wedge}\\
   &\quad +\{\alpha(x),\alpha(y),[z,w]^*\}-\{\alpha(y),\alpha(x),[z,w]^*\}-\Courant{x,y,z}\circ \alpha^2(w)+\alpha^2(w)\circ \Courant{x,y,z}\\
   &\quad -\Courant{x,y,z}\curlyvee \alpha^2(w)-\alpha^2(z) \circ \Courant{x,y,w}+\Courant{x,y,w}\circ \alpha^2(z)-\alpha^2(z)\curlyvee \Courant{x,y,w}\\
   &=[\alpha(x),\alpha(y),[z,w]^*]+\underline{\{\alpha(x),\alpha(y),z\circ w\}^{\wedge}}+\underline{\underline{\{\alpha(x),\alpha(y),w\circ z\}^{\wedge}}}+\{\alpha(x),\alpha(y),z\curlyvee w\}^{\wedge}\\
   &\quad +\underline{\underline{\underline{\{\alpha(x),\alpha(y),[z,w]^*\}}}}-\underline{\underline{\underline{\underline{\{\alpha(y),\alpha(x),[z,w]^*\}}}}}-\underline{\Courant{x,y,z}\circ \alpha^2(w)}\\
   &\quad +\alpha^2(w)\circ [x,y,z]+\underline{\underline{\alpha^2(w)\circ 
   \{x,y,z\}^{\wedge}}}+\underline{\underline{\underline{\alpha^2(w)\circ \{x,y,z\}}}}-\underline{\underline{\underline{\underline{\alpha^2(w)\circ \{y,x,z\}}}}}\\
   &\quad -\Courant{x,y,z}\curlyvee \alpha^2(w)-\alpha^2(z) \circ [x,y,w]-\underline{\alpha^2(z) \circ \{x,y,w\}^{\wedge}}-\underline{\underline{\underline{\alpha^2(z) \circ \{x,y,w\}}}}\\
   &\quad +\underline{\underline{\underline{\underline{\alpha^2(z) \circ \{y,x,w\}}}}}+\underline{\underline{\Courant{x,y,w}\circ \alpha^2(z)}}-\alpha^2(z)\curlyvee \Courant{x,y,w}=0.
\end{align*}
The other identities can be checked similarly. Thus $(A,[\cdot,\cdot]^*,\Courant{\cdot,\cdot,\cdot},\alpha)$ is a Hom-Lie-Yamaguti algebra.
\end{proof}
The following results illustrate that NS-Hom-Lie-Yamaguti algebras can be viewed as the underlying algebraic
structures of twisted $\mathcal{O}$-operators on Hom-Lie-Yamaguti algebras. 
\begin{thm}
    Let $T: V\rightarrow A$ a
    $(\huaF,\huaG)$-twisted $\mathcal{O}$-operator on the Hom-Lie-Yamaguti algebra with respect to a representation $(V,\rho,\theta,\beta)$. Then,    
    \begin{align}
    &  u\circ v= \rho(Tu)v, \quad u\curlyvee v=\huaF(Tu,Tv),\\ 
    &  \{u,v,w\}=\theta(Tv,Tw)u, \quad[u,v,w]=\huaG(Tu,Tv,Tw), \quad \forall u,v,w\in V,  
    \end{align}
    define a NS-Hom-Lie-Yamaguti algebra  structure on $V$. It is called the \textbf{induced NS-Hom-Lie-Yamaguti algebra}.
\end{thm}
\begin{proof}
For any $u_1,u_2,u_3,u_4 \in V$, 
\begin{align*}
    &\{\beta(u_1),\beta(u_2),u_3\circ u_4\}^{\wedge}-\beta^2(u_3)\circ \{u_1,u_2,u_4\}^{\wedge}-\Courant{u_1,u_2,u_3}\circ\beta^2(u_4)\\
    &=\{\beta(u_1),\beta(u_2),u_3\circ u_4\}^{\wedge}-\beta^2(u_3)\circ \{u_1,u_2,u_4\}^{\wedge}-[ u_1,u_2,u_3]\circ\beta^2(u_4)\\
    &\quad -\{u_1,u_2,u_3\}^{\wedge}\circ\beta^2(u_4)-\{u_1,u_2,u_3\}\circ\beta^2(u_4)+\{u_2,u_1,u_3\}\circ\beta^2(u_4)\\
    &=D_{\rho,\theta}(T\beta(u_1),T\beta(u_2))\rho(Tu_3)u_4-\rho(T\beta^2(u_3))D_{\rho,\theta}(Tu_1,Tu_2)u_4-\rho(T\huaG(Tu_1,Tu_2,Tu_3))\beta^2(u_4)\\
    &\quad -\rho(TD_{\rho,\theta}(Tu_1,Tu_2)u_3)\beta^2(u_4)-\rho(T\theta(Tu_2,Tu_3)u_1)\beta^2(u_4)+\rho(T\theta(Tu_1,Tu_3)u_2)\beta^2(u_4)\\
    &=D_{\rho,\theta}(\alpha(Tu_1),\alpha(Tu_2))\rho(Tu_3)u_4-\rho(\alpha^2(Tu_3))D_{\rho,\theta}(Tu_1,Tu_2)u_4-\rho([Tu_1,Tu_2,Tu_3])\beta^2(u_4) \overset{\eqref{RL7}}{=}0.
\end{align*}
\end{proof}
\begin{cor}
If $R:A \to A$ is a $(\lambda,\mu)$-weighted Reynolds operator on Hom-Lie-Yamaguti
algebra, then  
\begin{align*}
     & x\circ y= [Tx,y]\text, \quad x\curlyvee y=
      \lambda[Rx,Ry],\\ &\{x,y,z\}=\Courant{Rx,Ry,z},\quad[x,y,z]=\mu\Courant{Rx,Ry,Rz}. 
    \end{align*} 
    define a NS-Hom-Lie-Yamaguti algebra  structure on $A$.
\end{cor}
\begin{re}
    Twisted $\mathcal O$-operators (hence $(\lambda,\mu)$-weighted Reynolds operators) are related to Hom-NS-Lie-Yamaguti algebras in the same way as 
Rota-Baxter operators are related to Hom-pre-Lie-Yamaguti algebras.
\end{re}
\begin{cor}
The sub-adjacent  Hom-Lie-Yamaguti algebra $(V,[\cdot,\cdot]^*,\Courant{\cdot,\cdot,\cdot},\alpha)$ of the above NS-Hom-Lie-Yamaguti algebra $( V,\circ,\{\cdot,\cdot,\cdot\},\curlyvee,[\cdot,\cdot,\cdot],\alpha)$ is exactly the Hom-Lie-Yamaguti algebra given in Proposition \ref{PP}.
\end{cor}
\subsection{NS-Hom-Lie-Yamaguiti algebras arising from NS-Hom-Lie algebras}
We start by recalling the definition of NS-Hom-Lie algebra as particular case of NS-Hom-Leibniz algebras introduced in \cite{WangKe}.
\begin{defn}
    A NS-Hom-Lie algebra is a linear space $A$ together with bilinear operations $\circ,\curlyvee : A \rightarrow A$ and a linear map $\alpha:A\to A$ in which $\curlyvee$ is skew-symmetric and satisfying, for all $ x,y,z \in A$,
  \begin{align} \label{NS-Lie-1}
        &[x,y]\circ \alpha(z)- \alpha(x)\circ(y\circ z)+\alpha(y) \circ(x\circ z)=0,\\\label{NS-Lie-2}
    & \alpha(x) \curlyvee[y, z] + \alpha(y) \curlyvee [z,x] + \alpha(z) \curlyvee [x,z
    ]  +  \alpha(x)\circ( y  \curlyvee z) + \alpha(y)  \circ (z\curlyvee x)+ \alpha(z) \circ (x\curlyvee y) =0,  
\end{align}
where $ [x,y] =  x \circ y - y\circ x + x\curlyvee y.$
\end{defn}
\begin{re}
    If the bilinear operation $\circ$ in the above definition is trivial, one gets that $(A, \curlyvee,\alpha)$ is a Hom-Lie
algebra. On the other hand, if $\curlyvee$ is trivial, then $(A, \circ,\alpha)$ becomes a Hom-pre-Lie algebra. Thus, NS-Hom-Lie algebras
are a generalization of both Hom-Lie algebras and Hom-pre-Lie algebras.
\end{re}
\begin{prop}
    Under the above notations $[\cdot,\cdot]$ gives a Hom-Lie structure on $A$ with respect $\alpha$, called the adjacent Hom-Lie algebra of $A$.
\end{prop}
\begin{prop}
    Let $T:V \to A $ be a twisted $\mathcal{O}$-operator on  Hom-Lie algebra $(A, [\cdot,\cdot],\alpha)$ with respect to the representation $(V,\rho,\beta)$. Then
\begin{align}\label{induced-NS}
    &u\circ v=\rho(Tu)v \quad\text{and} \quad
    u\curlyvee v=\huaF(Tu,Tv)
\end{align}
defines a  \textsf{NS-Lie algebra} structure on $V$.
\end{prop}


Let
  $(A,\circ ,\curlyvee,\alpha )$ be a \textsf{NS}-Hom-Lie algebra. We define the two brackets 
  \begin{align}
     &\{x,y,z\} = \alpha(z) \circ (y \circ x),\\
     &[x,y,z]=  [x,y]\curlyvee \alpha(z) - \alpha(z) \circ (x\curlyvee y ) \ \ \forall x,y,z \in L.
 \end{align}
Let $[\![x,y,z]\!]=[[x,y],z]$.
\begin{thm}
Under the above notations   $( V,\circ,\{\cdot,\cdot,\cdot\},\curlyvee,[\cdot,\cdot,\cdot],\alpha)$ is a NS-Hom-Lie Yamaguti algebra.
 
\end{thm}
 \begin{proof}
 Let $x,y,z$ and $t  \in A$. First by  direct calculation we have   
    \begin{align*}
&\underset{x,y,z}{\circlearrowleft}\bigg([x,y]^*\curlyvee \alpha(z)-\alpha(x)\circ( y\curlyvee z)+[x,y,z]\bigg)\\&=\underset{x,y,z}{\circlearrowleft}\bigg([x,y]^*\curlyvee \alpha(z)-\alpha(x)\circ( y\curlyvee z)+[x,y]\curlyvee \alpha(z)-\alpha(z)\circ(x\curlyvee y)\bigg)\\&=2\underset{x,y,z}{\circlearrowleft}\bigg([x,y]^*\curlyvee \alpha(z)-\alpha(x)\circ( y\curlyvee z)\bigg)=0.
    \end{align*}
    Thus we have proved \eqref{NScocycle1}. Next, 
 we have 
    \begin{align*}
        &\{\alpha(x_4),\alpha(x_1),[x_2,x_3]\}-\alpha^2(x_2)\circ\{x_4,x_1,x_3\}+\alpha^2(x_3)\circ \{x_4,x_1,x_2\}\\
        &=\alpha([x_2,x_3])\circ(\alpha(x_1)\circ\alpha(x_4))-\alpha^2(x_2)\circ (\alpha(x_3)\circ(x_1\circ x_4))+\alpha^2(x_3)\circ (\alpha(x_2)\circ(x_1\circ x_4))
        \overset{\eqref{NS-Lie-1}}{=}0.
    \end{align*}
    Hence the identity \eqref{40} is proved.
    Now we check the identity \eqref{NScocycle3}. It follows from the NS-Hom-Lie identity \eqref{NS-Lie-2}:
      \begin{align*}
       & \{\alpha(x),\alpha(y),z\curlyvee t\}^\wedge+\alpha^2(t)\circ[x,y,z]-\Courant{x,y,z}\curlyvee\alpha^2(t)\\ 
       &\quad +[\alpha(x),\alpha(y),[z,t]^*]-\alpha^2(z)\circ [x,y,t]-\alpha^2(z)\curlyvee\Courant{x,y,t}\\
       &=\underline{[\alpha(x),\alpha(y)]\circ(z\curlyvee t)}+\underline{\alpha^2(t) \circ([x,y]\curlyvee \alpha(z))}-\underline{\underline{\alpha^2(t)\circ(\alpha(z)\circ(x\curlyvee y))}}-\underline{[[x,y],\alpha(z)]\curlyvee \alpha^2(t)}\\
       &\quad +\underline{[\alpha(x),\alpha(y)]\curlyvee \alpha([z,t]^*)}-\underline{\underline{\alpha([z,t]^*)\circ(\alpha(x)\curlyvee \alpha(y))}}-\underline{\alpha^2(z)\circ([x,y]\curlyvee\alpha(t))}\\
       &\quad +\underline{\underline{\alpha^2(z)\circ(\alpha(t)\circ(x\curlyvee y))}}-\underline{\alpha^2(z)\curlyvee [[x,y],\alpha(t)]}=0.
    \end{align*}
    The other identities are shown by direct computation and using \eqref{NS-Lie-1}-\eqref{NS-Lie-2}. 
    
    Thus, $( V,\circ,\{\cdot,\cdot,\cdot\},\curlyvee,[\cdot,\cdot,\cdot],\alpha)$ is a NS-Hom-Lie Yamaguti algebra. 
 \end{proof}

\end{document}